\documentclass[10pt,a4paper]{amsart}
\usepackage{rotating}
\usepackage[all]{xy}
\usepackage[utf8]{inputenc}
\usepackage{amscd,amssymb,amsopn,amsmath,amsthm,graphics,amsfonts,enumerate,verbatim,calc}

\numberwithin{equation}{section}

\newtheorem{theorem}{Theorem}[section]
\newtheorem{lemma}[theorem]{Lemma}
\newtheorem{proposition}[theorem]{Proposition}
\newtheorem{corollary}[theorem]{Corollary}
\newtheorem{notation}[theorem]{Notation}

\newtheorem{observation}[theorem]{Observation}

\theoremstyle{definition}

\newtheorem{definition}[theorem]{Definition}
\theoremstyle{remark}
\newtheorem{remark}[theorem]{Remark}
\newtheorem{fact}[theorem]{Fact}
\newtheorem{example}[theorem]{Example}
\newtheorem{question}[theorem]{Question}\newtheorem{discussion}[theorem]{Discussion}
\newtheorem{conjecture}[theorem]{Conjecture}

\newtheorem{acknowledgement}{Acknowledgement}

\newcommand{\Ass}{\operatorname{Ass}}
\newcommand{\Syz}{\operatorname{Syz}}

\newcommand{\Spec}{\operatorname{Spec}}

\newcommand{\Char}{\operatorname{Char}}

\newcommand{\id}{\operatorname{id}}

\newcommand{\pd}{\operatorname{pd}}

\newcommand{\Ext}{\operatorname{Ext}}

\newcommand{\Tor}{\operatorname{Tor}}
\newcommand{\Hom}{\operatorname{Hom}}

\newcommand{\depth}{\operatorname{depth}}

\newcommand{\im}{\operatorname{im}}

\newcommand{\lo}{\longrightarrow}\DeclareMathOperator{\fd}{fd}
\newcommand{\fm}{\mathfrak{m}}
\newcommand{\fp}{\mathfrak{p}}

\DeclareMathOperator{\Frac}{Frac}
\newcommand{\bx}{\boldsymbol x}
\newcommand{\bxt}{\boldsymbol x^{[t]}}
\newcommand{\LLambda}{\mathbf L\Lambda}
\newcommand{\Rlim}{\mathbf R\!\varprojlim}

\begin{document}
	
\author[M. Asgharzadeh and E. Mahdavi]{Mohsen Asgharzadeh and Elham Mahdavi}
	\date{}
\title[Non-Noetherian Bass and Betti numbers]{Non-Noetherian Bass and Betti numbers}
	
\address{M. Asgharzadeh}
	\email{mohsenasgharzadeh@gmail.com}
	
\address{E. Mahdavi}
\email{elham.mahdavi.gh@gmail.com}
	
	
\begin{abstract}
This paper investigates the vanishing and non-vanishing of Betti  and Bass numbers  for  non-finitely generated modules. We prove that for $d$-dimensional Cohen--Macaulay local rings, every non-zero $\mathfrak{m}$-torsion module satisfies $\beta_d(M)\neq 0$, and we establish the Betti number behavior of the injective hull $E_R(k)$. We study tor-rigidity for $H^d_{\mathfrak{m}}(R)$. We also provide  positive answer  to Schoutens' question on whether the vanishing of some Betti number  of a big Cohen--Macaulay algebra forces the Cohen--Macaulay property of $R$. For the absolute integral closure $R^+$, we establish both Tor and Ext results. On the Tor side, we prove that $\beta_i(R^+)=0$ for some $i>0$ implies regularity in a series cases. On the Ext side, we prove that $\mu_i(R^+)=0$ for some $i> d$ forces regularity for Gorenstein domains of prime characteristic, and we obtain analogous results for graded normal domains of dimension $2$ and also for quotient and isolated singularities in any dimension. Also $\mu_i(R^\infty)=0$ forces regularity for isolated singularity.
	\end{abstract}
	
	\maketitle
	
\section{Introduction}
Let $(R,\mathfrak{m},k)$ be a Noetherian local ring, and $M$ be an $R$-module.
The study of \textit{Betti numbers} $\beta_i^R(M)=\dim_k \operatorname{Tor}_i^R(k,M)$ and \textit{Bass numbers} $\mu^i_R(M)=\dim_k \operatorname{Ext}_R^i(k,M)$ is fundamental to commutative algebra, providing deep insights into the structure of modules and rings. For finitely generated modules over regular local rings, the Auslander--Buchsbaum formula and the theory of dualizing complexes completely govern the behavior of these invariants. However, when modules are not finitely generated, or when the base ring is singular, classical results often fail and new phenomena emerge. This paper systematically investigates the vanishing and non-vanishing of both Tor and Ext in these broader contexts, with particular emphasis on modules arising from local cohomology, injective hulls, and the absolute integral closure.
	
We place special emphasis on the Bass and Betti numbers of the following non-Noetherian objects:
$$
\big\{H^d_{\mathfrak m}(R),\ E_R(k),\ R^\infty,\ R^+\big\},
$$where $H^d_{\mathfrak m}(R)$ denotes the top local cohomology module of $R$ with respect to $\fm$,
and $E_R(k)$ is the injective hull of the residue field $k$. Also, the two fundamental infinite integral extensions in characteristic $p>0$ are the \emph{perfect closure} $R^\infty$—obtained by adjoining all $p$-power roots so that Frobenius becomes an isomorphism—and the \emph{absolute integral closure} $R^+$, the maximal integrally closed extension formed by adjoining roots of every monic polynomial (this construction is independent of characteristic). Originating in Artin's work~\cite{Ar}, $R^+$ was later employed by Hochster and Huneke~\cite{HH} as a source of big Cohen–Macaulay algebras when $\operatorname{char}(R)=p$. Today, both closures are central to tight closure theory, perfectoid geometry, and the study of singularities in positive and mixed characteristic. 
In equal characteristic zero, however, these objects have received comparatively little scrutiny. A key question, posed by Bhatt, Iyengar, and Ma, asks whether Tor-vanishing over $R^+$ forces regularity—a problem that remains open in general and serves as the main motivation for the homological perspective developed in this work.
	
	\begin{question}\label{Bim}(See \cite{Bim}).
	If $ R $ is a  domain of equicharacteristic zero   and $\beta_i(R^+)=0$ for some $i \geq 1$, then is $R$ regular?
\end{question}
	
	This question has attracted significant attention. In dimension $2$, Patankar \cite{p} proved that $\beta_i(R^+)=0$ forces regularity when a graded ring $R$ that contains $\mathbb{Q}$. 
	We provide affirmative answers for some classes of rings, including the following observation: \begin{theorem}\label{17}The following are valid:
		\begin{itemize}
			\item[(a)] If $k\subseteq R$  is a quotient singularity and $\beta_i(R^+)=0$ for some $i>0$, then $R$ is regular.
			\item[(b)] 
		
			Suppose $R$ is a homogeneous domain of dimension  over $\mathbb{C}$. If $R$ is Cohen--Macaulay and of finite Cohen--Macaulay type and $\beta_i(R^+)=0$ for some $i>0$, then 
			$R$ is a UFD and of multiplicity at most two. 
	\item[(c)] 	Suppose $R$ is a $d$-dimensional Cohen--Macaulay local ring containing $\mathbb{Q}$. If $d>2$, then $\beta_1(R^+)\neq 0$.	\end{itemize}
		
	\end{theorem}
Recall that item~(c) completes an argue initiated in \cite{p} by dealing with the remaining  case $d=3$.
Our investigation is organized around three interconnected themes (1)---(3):
	
(1) \textit{Non-vanishing of top Betti numbers:} For a $d$-dimensional local ring, when is $\beta_d(M) \neq 0$? We show that for Cohen--Macaulay rings, any non-zero $\mathfrak{m}$-torsion module has a non-zero $d$-th Betti number. This has direct consequences for the injective hull $E_R(k)$, for which we prove $\beta_d(E_R(k)) \neq 0$ and $\beta_i(E_R(k))=0$ for $i<d$, as was previously well-known over regular rings. These results extend classical non-vanishing theorems of Grothendieck. A sample result is to present a generalization to $\mathfrak{m}$-torsion modules.
		\begin{theorem}\label{13}
			Suppose $(R,\mathfrak{m})$ is a Cohen--Macaulay local ring of dimension $d$. Let $\Omega$ be an $\mathfrak{m}$-torsion module. Then $\Omega=0$  if and only if $\beta_d(\Omega)=0.$
		\end{theorem}
		
(2) \textit{Rigidity and local cohomology:} It is well known that the top local cohomology module $H^d_{\mathfrak{m}}(R)$ is a fundamental invariant encoding deep information about the singularities of $R$. We demonstrate that $H^d_{\mathfrak{m}}(R)$ is  {not} tor-rigid in general by showing that for a Cohen--Macaulay ring $R$,
		$
		\operatorname{Tor}_i^R(M, H^d_{\mathfrak{m}}(R)) \cong H^{d-i}_{\mathfrak{m}}(M),
		$
		and exploiting the depth of $M$. This provides a rich source of counterexamples to naive generalizations of rigidity and highlights the subtle interplay between Tor, local cohomology, and depth.
		
(3) \textit{The absolute integral closure $R^+$:} {Question} \ref{Bim} serves as a central motivating problem. We make significant progress from both Tor and Ext perspectives. 
A related problem, due to Schoutens, concerns the behavior of big Cohen--Macaulay algebras. Suppose $S$ is an $R$-algebra that is big Cohen--Macaulay and satisfies $\beta_i(S)=0$ for some  $i$. Schoutens asked in \cite{sc1} whether this forces $R$ to be Cohen--Macaulay. First, we provide two partial positive answers, and then inspired them to prove it in full generality:
	
	\begin{theorem}\label{14}
		Let $S$ be an $R$-algebra that is big Cohen--Macaulay.  Assume one of the following:
		\begin{itemize}
			\item[(i)] $R$ is an isolated singularity, and $\beta_i(S)=0$ for some $i \geq \dim(R)$.
			\item[(ii)] $S$ is weakly tor-rigid and $\beta_3(S)=0$.
		\end{itemize}
		Then $R$ is Cohen--Macaulay.
	\end{theorem}

	This result connects the vanishing of Betti numbers of big Cohen--Macaulay algebras to the Cohen--Macaulay property of the base ring, offering a new perspective on the relationship between homological dimensions and singularity theory. A sample tor-rigidity result we obtain over any $3$-dimensional ring $R$ and   a reflexive module $M$ is that
$\Tor_i^R(M,k)=0
		$
		  for some $ i$ implies that
		  $\Tor_j^R(M,k)=0$
 for all $j>i$. Then, we turn  to show:
	
\begin{corollary}\label{15new}
Let $M$ be a big Cohen--Macaulay $R$-module. If
$
		\Tor_n^R(k,M)=0\quad\text{for some }n\ge1,
$
		then $R$ is Cohen--Macaulay and
$\Tor_i^R(k,M)=0$  {for every } $i\ge1.$\end{corollary}
	
	In light of the Ext results above, we pose the following natural analogue of Question~\ref{Bim}: if $R$ is a local domain and $\mu_i(R^+)=0$ for some $i \geq \dim(R)+1$, must $R$ be regular? This cannot be settled by formal duality: although Matlis duality gives
	$
	\operatorname{Ext}_R^i(k, R^+)^\vee \cong \operatorname{Tor}_i^R(k, (R^+)^\vee),
	$
	the obstruction is that $(R^+)^\vee \ncong R^+$, see \cite[Corollary 4.12]{Bim}, where $H^+_{\mathfrak m}((R^+)^\vee)=0$ but $H^d_{\mathfrak m}(R^+)\neq 0$ at least in prime characteristic, $d>0$. Thus the dual question is not a formal consequence of Matlis duality and seems to demand new ideas.
We present the following:
\begin{theorem}	\label{18}\begin{itemize}
\item[(i)]  If $R$ is an $\mathbb{N}$-graded normal domain of dimension $2$ over an equicharacteristic zero field $k$, and $\mu_{i}(R^{+})=0$ for some $i> 2$, then $R$ is regular.
\item[(ii)] If $R$ is a local ring of dimension $d$ with a quotient and isolated singularity containing $\mathbb{Q}$, and $\mu_{i}(R^{+})=0$ for some $i>d$, then $R$ is regular.
			
\item[(iii)] 	Let $(R,\mathfrak{m})$ be a Gorenstein local ring  isolated singularity containing $\mathbb{Q}$. If $\mu_{i}(R^{+})=0$ for some $i>\dim(R)$, then $R$ is regular.
\item[(iv)] 	Let $\mathbb{F}_p\subseteq R$ be a $d$-dimensional ring with isolated singularity. If 
either	$
\mu_{i}(R^\infty) = 0$, or $
		\mu_{i}(R^+) = 0$
for some $i > d$, then $R$ is regular.
\item[(v)] 	Suppose $\mathbb{Q}\subseteq R$ has an isolated singularity,  $(S_2)$, and $\dim R\ge 4$. If $R$ is  almost complete intersection and $\operatorname{Ext}_R^i(k,R^+)=0$ for some $i>d$, then $R$ is a UFD.
			
				\end{itemize}
	\end{theorem}

The paper is organized as follows. Section 2 collects preliminary counterexamples, showing, for instance, that the tensor product of two non-generalized Cohen--Macaulay modules can be generalized Cohen--Macaulay. Then we present our main results on tensor products and generalized Cohen--Macaulayness, extending classical theorem  of Huneke--Wiegand. 
	We close this section with the corresponding property of big Cohen-Macaulay modules, and the behavior of $\beta_1(-)$ under tensor product. 
Subsection 3.1 presents the proof of {Theorem} \ref{14}.
	Subsection 3.2 provides a detailed analysis of $\beta_1(H^1_{\mathfrak{m}}(R))$ for a specific non-Cohen--Macaulay ring, illustrating the failure of rigidity.
In Section 4, we prove {Theorem} \ref{13} plus {Corollary} \ref{15new}, and explore consequences for the injective hull $E_R(k)$, showing $\beta_d(E_R(k))\neq 0$ while lower Betti numbers vanish.  Also, it contains further results on Matlis reflexive modules and the relationship between Betti numbers and local cohomology.
	Section 5 is devoted to the proof of theorems  \ref{17} and \ref{18}.

\begin{notation}	Throughout this paper, all rings are commutative   with identity, and all modules are unital. For a noetherian  and local ring $(R,\mathfrak{m},k)$, we denote by $E_R(k)$ the injective hull of the residue field. Also, $d$ is used for Krull dimension, otherwise specialized. For an $R$-module $M$, we write $\beta_i^R(M)=\dim_k \operatorname{Tor}_i^R(k,M)$ for the $i$-th Betti number and $\mu^i_R(M)=\dim_k \operatorname{Ext}_R^i(k,M)$ for the $i$-th Bass number. The $i$-th local cohomology module with support in $\mathfrak{m}$ is denoted by $H^i_{\mathfrak{m}}(M)$. The absolute integral closure of a domain $R$ is denoted by $R^+$.
	\end{notation}

	\section{Tensor product and  Cohen--Macaulay variations}
	
In this section, we collect several results addressing the following natural question: if the tensor product $M \otimes_R N$ is (generalized, or even big) Cohen--Macaulay, when can one conclude that each of the factors $M$ and $N$ enjoys the same property? This question was raised in   \cite[Question 1.1]{cel}:

	\begin{question}
		Let $R$ be a complete intersection. Suppose $M\otimes_R N$ is maximal Cohen--Macaulay. When does it follow that $M$ or $N$ is maximal Cohen--Macaulay?
	\end{question}
	
	\begin{remark}(See \cite[Theorem 1.4]{cel}).
		A partial answer is suggested by the following: if $M$ is maximal Cohen--Macaulay and $\dim R > \operatorname{codim}(R)$, then $N$ is maximal Cohen--Macaulay.
	\end{remark}
	
The next example demonstrates that the generalized Cohen--Macaulay property is not preserved under tensor products, in the sense that both factors need not be generalized Cohen--Macaulay even if their tensor product is.

\begin{example}
	Let $R = k[[x_1,\ldots,x_n, y_1,\ldots,y_n]]$ with $n \ge 3$. Set $M = R/I$ with
	$$
	I = (x_1,\ldots,x_n) + (y_1) \cap (y_2, y_3),
	$$
	and $N = R/J$ with
	$$
	J = (y_1,\ldots,y_n) + (x_1) \cap (x_2, x_3).
	$$
	Then
	$
	M \otimes_R N = R/(I + J) = R/\mathfrak{m},
	$
	which is generalized Cohen--Macaulay.
The module $M$ is not equidimensional, as there exists
	$$
	\mathfrak{p} \in \operatorname{Ass}(M) = \operatorname{Ass}(M) \setminus \{\mathfrak{m}\}
	= \{(x_1,\ldots,x_n) + (y_1),\ (x_1,\ldots,x_n) + (y_2, y_3)\}
	$$
	such that $\dim(R/\mathfrak{p}) \neq \dim(R/I)$. Similarly, $N$ is not equidimensional on the punctured spectrum.
By \cite[Example 9.5.7(1)]{BH}, if $M$ were generalized Cohen--Macaulay, then $\dim(R/\mathfrak{p}) = \dim(M)$ for all $\mathfrak{p} \in \operatorname{Ass}(M) \setminus \{\mathfrak{m}\}$, and $M_{\mathfrak{p}}$ would be Cohen--Macaulay. Hence $M$ and $N$ are not generalized Cohen--Macaulay, but their tensor product $M \otimes_R N = R/\mathfrak{m}$ is.
\end{example}

\begin{proposition}
	Suppose $R$ is a complete Cohen--Macaulay domain with an isolated singularity. Let $M$ and $N$ be $R$-modules with $\dim M = \dim N = \dim R$. If $M \otimes_R N$ is generalized Cohen--Macaulay, then both $M$ and $N$ are generalized Cohen--Macaulay.
\end{proposition}

\begin{proof}
	Since $R$ is a domain and both modules have full dimension, we have $\operatorname{Supp}(M) = \operatorname{Supp}(N) = \operatorname{Spec}(R)$. For any non-maximal prime ideal
	$$
	\mathfrak{p} \in \operatorname{Supp}(M \otimes_R N) \setminus \{\mathfrak{m}\} = \operatorname{Spec}(R) \setminus \{\mathfrak{m}\},
	$$
	the localization
	$
	(M \otimes_R N)_{\mathfrak{p}} \cong M_{\mathfrak{p}} \otimes_{R_{\mathfrak{p}}} N_{\mathfrak{p}}
	$
	is Cohen--Macaulay over $R_{\mathfrak{p}}$ by \cite[Ex.~9.5(1)]{BH}.	
	Recall that $R_{\mathfrak{p}}$ is regular, and hence any maximal Cohen--Macaulay $R_{\mathfrak{p}}$-module is free. Thus,
	$
	(M \otimes_R N)_{\mathfrak{p}} \cong M_{\mathfrak{p}} \otimes_{R_{\mathfrak{p}}} N_{\mathfrak{p}}
	$
	is free. It follows that both $N_{\mathfrak{p}}$ and $M_{\mathfrak{p}}$ are free (see \cite[3.4.7]{BJ}); hence they are Cohen--Macaulay.
	Now take any $\mathfrak{p} \in \operatorname{Ass}(M) \setminus \{\mathfrak{m}\}$. We see that $\mathfrak{p} = (0)$. Applying \cite[Ex.~9.5.7(b)]{BSh}, we obtain that
	$
	\dim(R/\mathfrak{p}) = \dim M
	$
	for every $\mathfrak{p} \in \operatorname{Min}(\operatorname{Supp}(M)) \setminus \{\mathfrak{m}\}$, and that $M_{\mathfrak{p}}$ is Cohen--Macaulay for every $\mathfrak{p} \in \operatorname{Supp}(M) \setminus \{\mathfrak{m}\}$. These conditions precisely characterize $M$ as generalized Cohen--Macaulay. By symmetry, the same conclusion holds for $N$.
\end{proof}

	\begin{proposition}
		Suppose $R$ is a complete hypersurface domain and $M,N$ are $R$-modules with $\dim M=\dim N=\dim R$. If $M\otimes_R N$ is generalized Cohen--Macaulay, then both $M$ and $N$ are generalized Cohen--Macaulay.
	\end{proposition}
	
	\begin{proof}
		The argument closely parallels the preceding observation, with the isolated singularity hypothesis replaced by a theorem of Huneke--Wiegand \cite[3.1]{tensor}: over a hypersurface ring, if $M\otimes_R N$ is Cohen--Macaulay, then both $M$ and $N$ are Cohen--Macaulay. Since $M\otimes_R N$ is generalized Cohen--Macaulay, for every non-maximal prime $\mathfrak{p}\in\Spec(R)\setminus\{\mathfrak{m}\}$, the localization $(M\otimes_R N)_{\mathfrak{p}}$ is Cohen--Macaulay over the local hypersurface ring $R_{\mathfrak{p}}$. Applying the Huneke--Wiegand result locally yields that both $M_{\mathfrak{p}}$ and $N_{\mathfrak{p}}$ are  Cohen--Macaulay. The remainder of the proof follows identically to the previous case: one shows that $\Ass(M)\setminus\{\mathfrak{m}\}=\{(0)\}$ and similarly for $N$, after which the criterion from \cite[Ex.~9.5.7(b)]{BSh} ensures that $M$ and $N$ are generalized Cohen--Macaulay.
	\end{proof}
	
\begin{proposition}
	Let $(R,\mathfrak{m})$ be a regular local ring and let $M\otimes_R N$ be a balanced big Cohen--Macaulay $R$-module.
	Then $M$ is balanced big Cohen--Macaulay if and only if $N$ is.
\end{proposition}

\begin{proof}
	Since $R$ is regular, every balanced big Cohen--Macaulay module is flat (see \cite[Page 61]{h}). Thus $M\otimes_R N$ is faithfully flat. 
By symmetry, it suffices to assume $M$ is balanced big Cohen--Macaulay. In particular, $M$ is faithfully flat. 
	To show that $N$ is flat, let $0 \to A \to B$ be an injection of $R$-modules. We need to prove that
$
	0 \longrightarrow A\otimes_R N \longrightarrow B\otimes_R N
$
	is exact. Tensoring the injection $0 \to A \to B$ with $M\otimes_R N$ gives
	$$
	0 \longrightarrow A \otimes_R (M\otimes_R N) \longrightarrow B \otimes_R (M\otimes_R N),
	$$
	which, by associativity of the tensor product, is equivalently
	$$
	0 \longrightarrow (A\otimes_R N)\otimes_R M \longrightarrow (B\otimes_R N)\otimes_R M.
	$$
	Since $M$ is faithfully flat, it follows that
$
	0 \longrightarrow A\otimes_R N \longrightarrow B\otimes_R N
$
	is exact. Hence $N$ is flat.
It remains to show faithful flatness. Suppose $N\otimes_R L = 0$ for some $R$-module $L$. Then
$$
	(M\otimes_R N)\otimes_R L = M\otimes_R (N\otimes_R L) = M\otimes_R 0 = 0.
	$$
Since $M\otimes_R N$ is faithfully flat, we conclude $L = 0$. Thus $N$ is faithfully flat.
	Finally, because $R$ is regular, every faithfully flat module is balanced big Cohen--Macaulay (see \cite[Page 61]{h}). Therefore $N$ is balanced big Cohen--Macaulay. By symmetry, the converse holds as well.
\end{proof}

\begin{example}
	Flatness of $M\otimes_R N$ need not imply flatness of either factor individually. Let $R$ be a domain with fraction field $Q$, and take $0 \neq r \in R$. For $M = Q \oplus R/rR$ and $N = Q$, we have
	$
	M \otimes_R N \cong Q,
	$
	which is flat over $R$. Yet $M$ is not flat, since its torsion summand $R/rR$ is not flat over the domain $R$.
\end{example}
\begin{proposition} 	\begin{itemize}
		\item[(i)]
	Suppose $A$ and $B$ are Tor-independent $R$-modules such that $\beta_1(A)=\beta_1(B)=0$. Then $\beta_1(A\otimes_R B)=0$.
	
	\item[(ii)]	
	Suppose the pair $(B,C)$ is Ext-independent and that $\mu_1(C)=\beta_1(B)=0$. Then $\mu_1(\Hom( B,C))=0$.
\end{itemize}\end{proposition}

\begin{proof}
(i)	Recall that Tor-interdependency means that $\Tor^R_i(A,B)=0$ for all $i>0$.
Since $A$ and $B$ are Tor-independent, there is a spectral sequence
	$$
	\operatorname{Tor}_p^R(A,\operatorname{Tor}_q^R(B,k)) \Longrightarrow \operatorname{Tor}_{p+q}^R(A\otimes_R B,k),
	$$
	(see \cite[Theorem 10.59]{Rot}). Again by \cite[Theorem 10.31]{Rot}, there is an exact sequence
	$$
	\operatorname{Tor}_2^R(A\otimes_R B,k) \to \operatorname{Tor}_2^R(A,B\otimes_R k) \to A\otimes_R \operatorname{Tor}_1^R(B,k) \to \operatorname{Tor}_1^R(A\otimes_R B,k) \to \operatorname{Tor}_1^R(A,B\otimes_R k) \to 0.
	$$
	Now $B\otimes_R k \cong k^{\,t}$ for some $t\in[0,\infty]$. Consequently $\operatorname{Tor}_1^R(A,B\otimes_R k)=0$. From the displayed exact sequence, $\beta_1(A\otimes_R B)=0$.

(ii)
	Since $(B,C)$ is Ext-independent, 
$\operatorname{Ext}_R^i(B \otimes_R P, C) = \{0\}
$
	for all $i \ge 1$ whenever ${}_R P$ is projective. 
	Then there is a third quadrant spectral sequence (see \cite[Theorem 10.62]{Rot}):
	$$
	\operatorname{Ext}_R^p\bigl(\operatorname{Tor}_R^q(B, k), C\bigr) 
	\Rightarrow_p \operatorname{Ext}_R^n\bigl(k, \operatorname{Hom}_R(B, C)\bigr).
	$$
	  By \cite[Theorem 10.33]{Rot}, there is an exact sequence
		$$
	0 \longrightarrow E_2^{1,0} \longrightarrow H^1(\operatorname{Tot}(M)) 
	\longrightarrow E_2^{0,1} \xrightarrow{d_2} E_2^{2,0} 
	\longrightarrow H^2(\operatorname{Tot}(M)).
	$$
	Now $B\otimes_R k \cong k^{\,t}$ for some  $t\in[0,\infty]$. Consequently 
	$E_2^{1,0}=\operatorname{Ext}_R^1\bigl(\operatorname{Tor}_R^0(B, k), C\bigr)=\operatorname{Ext}_R^1\bigl(k^{\,t}, C\bigr)=0$. Similarly,  
	$E_2^{0,1}=\operatorname{Ext}_R^0\bigl(\operatorname{Tor}_R^1(B, k), C\bigr)=\operatorname{Hom}_R \bigl(0, C\bigr)=0$.
 From the displayed exact sequence, $\mu_1(\operatorname{Hom}_R(B, C))=0$.
\end{proof}

	\section{Weakly tor-rigid}
This section is divided into two subsections:
	\subsection{Weak tor-rigidity and big Cohen--Macaulay algebras}
	
	This subsection focuses on tor-rigidity and big Cohen--Macaulay algebras, proving that Schoutens' question has a positive answer for tor-rigid algebras.
	
\begin{question}(Schoutens, \cite[End of \S 2]{sc1}).
	Suppose $S$ is a big Cohen--Macaulay algebra over a local ring $R$. If $\operatorname{Tor}_n^R(k,S)=0$ for some $n\ge 1$, does it follow that $R$ is Cohen--Macaulay?
\end{question}First, recall that vanishing of all Betti numbers does not imply flatness:
\begin{example}[Bartijn--Strooker \cite{BS}]
	Let $R = k[[x,y,z]]$ and let $\varphi$ be a free module of infinite countable rank. Define $\Omega := \varphi + (x,y)\widehat{\varphi}$. Then $\{x,y,z\}$ is an $\Omega$-sequence, $\{z,x,y\}$ is not an $\Omega$-sequence, and $\operatorname{Tor}_i^R(k,\Omega)=0$.  
\end{example}
The following extends Example 3.2, and simplifies a result of Huneke (see \cite[9.1]{Hun}).

\begin{observation}\label{3.3}
	If $R$ is a regular local ring and $M$ is a big Cohen--Macaulay module, then $\beta_i(M)=0$ for all $i>0$.
\end{observation}

\begin{proof}
	Since $R$ is regular, the residue field $k$ admits a finite free resolution
	$$
	0 \longrightarrow F_d \longrightarrow F_{d-1} \longrightarrow \cdots \longrightarrow F_0 \longrightarrow k \longrightarrow 0.
	$$
	Tensoring this resolution with $M$ yields a complex whose homology is precisely $\operatorname{Tor}_i^R(k,M)$. Because $M$ is big Cohen--Macaulay, the Koszul complex on any system of parameters is exact on $M$, which is equivalent to the vanishing of all positive $\operatorname{Tor}$'s with the residue field; hence $\operatorname{Tor}_i^R(k,M)=0$ for every $i>0$. Consequently, the Betti numbers $\beta_i(M)=\dim_k \operatorname{Tor}_i^R(k,M)$ vanish for all $i>0$.
\end{proof}

\begin{definition}
	An $R$-module $S$ is called \emph{weakly tor-rigid} if $\operatorname{Tor}_j^R(R/I,S)=0$ for some parameter ideal $I$ and some integer $j$ implies $\operatorname{Tor}_{j+1}^R(R/I,S)=0$.
\end{definition}
Concerning Question 3.1, he proved the desired claim when $n\leq 2$. Here, we deal with the case $n=3$, and the general case seems being analogous.
\begin{theorem}
Suppose $S$ is a big Cohen--Macaulay, weakly tor-rigid and  $\operatorname{Tor}_3^R(S,k)=0$.   Then $R$ is Cohen--Macaulay.
\end{theorem}

\begin{proof}
 There exists an exact sequence
$
	0 \to M \to F_1 \xrightarrow{\varphi} F_0 \to  S \to 0,
$
	where $M=\Syz_2(S)$ and $F_0,F_1$ are free modules (not necessarily of finite rank). As observed in \cite[2.5]{sc1}, it suffices to prove that $IS\cap R=I$ for some parameter ideal $I$. 
Consider the two short exact sequences arising from the free resolution:
	$$
	0 \longrightarrow M \longrightarrow F_1 \longrightarrow E:=\operatorname{Im}(\varphi) \longrightarrow 0,
	$$
	$$
	0 \longrightarrow E \longrightarrow F_0 \longrightarrow S \longrightarrow 0.
	$$
	Tensoring these with $\overline{R}:=R/I$ yields the exact sequences

\begin{itemize}
	\item[(i)]
 $0 \longrightarrow \operatorname{Tor}_1^R(E,\overline{R}) \longrightarrow M/IM \longrightarrow F_1\otimes_R \overline{R} \longrightarrow E\otimes_R \overline{R} \longrightarrow 0,$ and
 \item[(ii)]	
	$	0 \longrightarrow \operatorname{Tor}_1^R(S,\overline{R}) \longrightarrow E\otimes_R \overline{R} \longrightarrow F_0\otimes_R \overline{R} \longrightarrow S\otimes_R \overline{R} \longrightarrow 0.
	$ 
\end{itemize}

Now, tensoring the short exact sequence $0\to I\to R\to R/I\to 0$ with $S$ to obtain
$
	\operatorname{Tor}_1^R(R/I,S) \to S\otimes_R I \to S.
$
	Since $S$ is big Cohen-Macaulay, $I$ is generated by an $S$-sequence. Hence, the map $S\otimes_R I\to S$ is injective, so $\operatorname{Tor}_1^R(R/I,S)=0$. Moreover, the natural map $I\otimes_R S\to IS$ is an isomorphism. Applying tor-rigidity, we also get $\operatorname{Tor}_2^R(R/I,S)=0$. Assuming $\operatorname{Tor}_1^R(R/I,E)=\operatorname{Tor}_2^R(R/I,S)=0$, the first tensored sequence reduces to
	$$
\zeta:=	0\longrightarrow M/IM \longrightarrow F_1\otimes_R \overline{R} \longrightarrow E\otimes_R \overline{R} \longrightarrow 0.
	$$
	From this we deduce $\operatorname{Tor}_1^R(M,k)=\operatorname{Tor}_3^R(S,k)=0$, and consequently $\operatorname{Tor}_1^{\overline{R}}(M/IM,k)=0$ by \cite[2.1]{sc1}. Hence $M/IM$ is flat over $\overline{R}$, (see \cite[22.3]{mat}). Now, by $\zeta$	we deduce $E\otimes_R \overline{R}$ has finite flat dimension over $\overline{R}$.
	Since $\overline{R}$ is zero-dimensional, the flat dimension of $E\otimes_R \overline{R}$ over $\overline{R}$ is at most dimension of the ring, then it is actually flat over the artinian ring $\overline{R}$. 
By 		$$
0 = \operatorname{Tor}_1^R(S,\overline{R}) \longrightarrow E\otimes_R \overline{R} \longrightarrow F_0\otimes_R \overline{R} \longrightarrow S\otimes_R \overline{R} \longrightarrow 0.
$$
we see $S\otimes_R \overline{R}$
is actually flat over   $\overline{R}$.  Recall that flat modules are torsion-free, thus, annihilator
of $S\otimes_R \overline{R}=S/IS$
as an $R$-module is $I$.
In other words, $IS\cap R=I$. This ensures that $I$ is generated by an $R$-sequence, as
$I$ is $S$-regular sequence.
By definition, $R$ is Cohen--Macaulay.
\end{proof}

	\begin{remark}
		If $R$ is Cohen--Macaulay and $M$ is a big Cohen--Macaulay module, then $M$ is weakly tor-rigid. To see this, let $\mathbf{x}=x_1,\ldots,x_n$ be a system of parameters for $R/I$. Since $M$ is big Cohen--Macaulay, $\mathbf{x}$ is an $M$-sequence. The Koszul complex $K_\bullet(\mathbf{x};M)$ is acyclic. Therefore
$
		\operatorname{Tor}_j^R(M,R/I) = H_j(M\otimes K_\bullet(\mathbf{x};R)) = H_j(\mathbf{x};M)=0
$
		for all $j>0$. The Cohen--Macaulayness of $R$ ensures that every system of parameters is regular on $R$, which identifies $R/I$ with the quotient by a regular sequence.
	\end{remark}

	\begin{fact}\label{shh1}(See \cite[Theorem VI.10]{sc3}).	Suppose there exists $a \ge 1$ such that
$
		\beta_i^R(\mathfrak{p}, -)=0
	$
		for all $\mathfrak{p} \in \mathrm{Sing}(R)$ and $0 \le i \le \mathrm{ht}(\mathfrak{p})$.
		Then $\mathrm{pd}_R(-)$ is finite.\end{fact}

By residual dimension we  mean:	 $\mathrm{resdim}(S):=\sup\{i:\beta_i(S)\neq 0\}.$ 
	The following completes the proof of {Theorem} \ref{14}:
	\begin{theorem}\label{thm:isolated}
		Assume $R$ is an isolated singularity and $S$ is an $R$-algebra such that $S$ is big Cohen--Macaulay and $\operatorname{Tor}_i^R(S,k)=0$  for some $i \geq \dim(R)$. Then $R$ is Cohen--Macaulay.
	\end{theorem}
	
\begin{proof}
By \cite[1.1]{IYENGAR}, $\operatorname{Tor}_j^R(S,k)=0$ for all $j \ge i$. Now recall from {Fact} \ref{shh1}
that $\mathrm{pd}_R(S) < \infty$, here we may pass to some syzygies of $S$, but this has no effect on finiteness of projective dimension. These enable  us to use a result of  Bartijn and  Strooker \cite[4.1]{BS} (also, see the  updated version presented in \cite[7.4]{sc2}). This implies that $\mathrm{depth}(R) - \mathrm{depth}(S) = \mathrm{resdim}(S)$. But $S$ is big Cohen--Macaulay, so $\mathrm{depth}(S)=\dim(R)$. Hence,
$$0 \le \mathrm{resdim}(S) = \mathrm{depth}(R)-\mathrm{depth}(S) \le \dim(R)-\dim(R)=0.
$$
Thus $\mathrm{resdim}(S)=0$. Therefore $\operatorname{Tor}_i^R(S,k)=0$ for all $i>0$.  In particular, $\operatorname{Tor}_1^R(S,k)=0$.
By \cite[Proposition 2.5]{sc1}, $R$ is Cohen--Macaulay.
\end{proof}

\begin{proposition}
Suppose $R$ is a $3$-dimensional local ring and $M$ is a reflexive module such that $\operatorname{Tor}_i^R(M,k)=0$ for some integer $i>0$. Then $\operatorname{Tor}_j^R(M,k)=0$ for all $j>i$.
\end{proposition}
	
\begin{proof}
We may assume $M\neq 0$.	Every reflexive module over a local ring is canonically isomorphic to its double dual, so we may write $M \cong N^*$ for some module $N$.   Assume without loss of generality that $M^*\neq 0$. Take a partial free resolution $F_1\to F_0\to N\to 0$, where $F_0$ and $F_1$ are  free modules not necessarily of finite rank. Since the dual of a free module is flat (indeed, $(\oplus R)^*\cong \prod R$ is flat  over any coherent rings, and recall that in our setup $R$  is noetheian and so coherent), dualizing the resolution yields an exact sequence
$0 \to N^* \to F_0^* \to F_1^* \to\operatorname{Tr}(N) \to 0,
$
where $\operatorname{Tr}(N)$ denotes the Auslander transpose of $N$. This gives isomorphisms $\operatorname{Tor}_{j+2}^R(\operatorname{Tr}(N),k)\cong \operatorname{Tor}_j^R(N^*,k)$ for all $j\ge 0$. Recall $N^\ast= M^{\ast\ast}=M$. From $\operatorname{Tor}_i^R(M,k)=0$, we obtain $\operatorname{Tor}_{i+2}^R(\operatorname{Tr}(N),k)=0$. Since $\dim R=3$, we have $i+2\ge \dim R$, so the result of \cite[1.1]{IYENGAR} applies to yield $\operatorname{Tor}_j^R(\operatorname{Tr}(N),k)=0$ for all $j\ge i+2$. Translating back, we get $\operatorname{Tor}_j^R(N^*,k)=0$ for all $j\ge i$. Since $N^\ast= M^{\ast\ast}=M$, and therefore $\operatorname{Tor}_j^R(M,k)=0$ for all $j\ge i$.
	\end{proof}

\begin{remark}	One has $\Tor_i^R(\prod_j M_j,k) =\prod_j\Tor_i^R( M_j,k) $, because noeherian rings are coherent (see \cite[3.2.26]{en}). Thus, if for some $j_0$ we know  $\beta_i(M_{j_0})$ is nonzero, then  $\beta_i(\prod_jM_j)\neq 0$.
\end{remark}

\begin{proposition}
Let $M$ be a torsionless module over a $d$-dimensional local ring $R$. If $\operatorname{Tor}_i^R(k,M)=0$ for some $i\ge d-1$, then $\operatorname{Tor}_j^R(k,M)=0$ for all $j\ge d$.
\end{proposition}

\begin{proof}
Let $\oplus R\to M^\ast\to 0$ be exact. Then $M\subseteq M^{\ast\ast}\hookrightarrow(\oplus R)^\ast=\prod R=:F$, 	gives an embedding $0\to M\to F$ into the flat module $F$. Let $X$ denote the cokernel, so we have a short exact sequence
$0 \to M \to F \to X \to 0.$
Applying $\operatorname{Tor}_*^R(k,-)$ and using the flatness of $F$, we obtain isomorphisms $\operatorname{Tor}_{n+1}^R(k,X)\cong \operatorname{Tor}_n^R(k,M)$ for all $n\ge 0$. In particular, $\operatorname{Tor}_d^R(k,X)\cong \operatorname{Tor}_{d-1}^R(k,M)=0$, where the vanishing follows from the hypothesis since $i\ge d-1$. By the result of \cite{IYENGAR}, the vanishing of $\operatorname{Tor}_d^R(k,X)$ forces $\operatorname{Tor}_{d+n}^R(k,X)=0$ for all $n\ge 0$. Translating back via the isomorphism above yields $\operatorname{Tor}_{d+n-1}^R(k,M)=0$ for all $n\ge 0$, which is precisely the desired vanishing for all $j\ge d$.
\end{proof}

\begin{remark}
The torsionless assumption is important. Indeed, suppose $R$ is regular of dimension $d>0$. Then $E_R(k)$ satisfies $\beta_{d-1}(E_R(k))=0$, while $\beta_{d}(E_R(k))\neq 0$. 
\end{remark}

 We need the following:

\begin{fact}\label{prop:CIM}
	Let $N$ be an $R$-module, and let $c\ge0$ be an integer such that
$H_{\fm}^i(N^\vee)=0$ for all $i>c$.
If $\Tor_n^R(k,N)=0$ for some $n\ge\max\{1,c\}$, then
$\sup\{i:\Tor_i^R(k,N)\ne0\}=\depth R-\depth_R N,$
	and $\Tor_i^R(k,N)=0$ for every $i\ge n$.
\end{fact}

This is the module case of \cite[Proposition~3.3(2)]{IYENGAR}:
the stated bound implies
$n\ge\sup\{i:H_{\fm}^i(N^\vee)\ne0\}$.
We will apply it when $N/\fm N\ne0$, so the above supremum   is at least zero. The proposition does not require that
$N$ have finite flat dimension over $R$.

\begin{lemma}\label{lem:bounds}
	Let $N$ be an $R$-module. Suppose that $x_1,\ldots,x_d$ is a system of
	parameters of $R$ and that $x_1,\ldots,x_r$ is $N$-regular, where
	$0\le r\le d$. Then
$	\depth_R N\ge r$
and $	H_{\fm}^i(N^\vee)=0$ for $i>d-r.$
\end{lemma}

\begin{proof}
	First, positive powers preserve regularity, including for modules that
	are not finitely generated. For example, if $x_1$ is $N$-regular, then
	the successive factors of the filtration of $N/x_1^aN$ by powers of
	$x_1$ are isomorphic to $N/x_1N$. A sequence regular on both ends of a
	short exact sequence is regular on the middle module; the exact
	sequence remains exact on passage to the successive quotients.
	Applying this observation along that filtration, and then inducting
	on the length of the sequence, shows that
	$x_1^t,\ldots,x_r^t$ is $N$-regular for every $t\ge1$.
	For completeness, $	\depth_R N\ge r$ follows by induction on $r$
	from the long exact Ext sequence of
	$0\to N\xrightarrow{x_1}N\to N/x_1N\to0$.
	Every element of $\fm$ acts as zero on $\Ext_R^i(k,N)$.
	Consequently, vanishing of $\Ext_R^i(k,N/x_1N)$ in degrees less than
	$r-1$ implies vanishing of $\Ext_R^i(k,N)$ in degrees less than $r$.
	Set $\bxt=(x_1^t,\ldots,x_d^t)$.
	Resolving the regular initial segment first gives a quasi-isomorphism
$
	K_\bullet(\bxt;N)
	\simeq
	K_\bullet\bigl(x_{r+1}^t,\ldots,x_d^t;
	N/(x_1^t,\ldots,x_r^t)N\bigr).
$
	The complex on the right has length $d-r$. Hence
	$H_i(\bxt;N)=0$ for $i>d-r$.
	Finite freeness of the Koszul complex and injectivity of $E$ yield
$	H^i(\bxt;N^\vee)
	\cong \Hom_R\bigl(H_i(\bxt;N),E\bigr)=0$  for all $i>d-r$.
	Finally, local cohomology is the direct limit of these Koszul
	cohomology modules. Since $\sqrt{(x_1,\ldots,x_d)}=\fm$, this proves
$	H_{\fm}^i(N^\vee)=0$ for $i>d-r.$ 
\end{proof}

The following answers Schoutens' question in the full setting:
\begin{corollary}
	Let $M$ be a big Cohen--Macaulay $R$-module. If
	$
	\Tor_n^R(k,M)=0\quad\text{for some }n\ge1,
	$
	then $R$ is Cohen--Macaulay and
	$\Tor_i^R(k,M)=0$  {for every } $i\ge1.$\end{corollary}

\begin{proof} 
	Choose an $M$-regular system of parameters. Lemma~\ref{lem:bounds},
	with $r=d$, gives
$	\depth_R M\ge d,$ and  $H_{\fm}^+(M^\vee)=0$.
Fact~\ref{prop:CIM} therefore applies with $c=0$ to every
	positive vanishing degree. Recall the following
$s:=\sup\{i:\Tor_i^R(k,M)\ne0\}.$
	The defining condition $M/\fm M\ne0$ gives
	$\Tor_0^R(k,M)\ne0$, so $s\ge0$. Formula from  Proposition \ref{prop:CIM} now gives
$0\le s=\depth R-\depth_R M \le\depth R-d\le0.$
	Thus $\depth R=d$ and $s=0$, as required.
\end{proof}

\begin{corollary}\label{cor:nonvanishing}
	If $R$ is not Cohen--Macaulay, then every big Cohen--Macaulay
	$R$-module $M$ satisfies
$
	\beta_+^R(M)\ne0. 
$
\end{corollary}

\begin{lemma}\label{lem:completion}
	For a big Cohen--Macaulay $R$-module $M$, the natural map from $M$ to
	its completion identifies
$ 
		\LLambda^{\fm}(M)\simeq\widehat M
$ 
	in the derived category. Moreover, for every $i\ge0$ it induces
$\Tor_i^R(k,M)\cong\Tor_i^R(k,\widehat M).$ 
	In particular, $\widehat M/\fm\widehat M\cong M/\fm M\ne0$.
\end{lemma}

\begin{proof}
	Choose an $M$-regular system of parameters $\bx$ and set
	$\mathfrak q_t=(x_1^t,\ldots,x_d^t)$.
	Derived completion can be computed as
$
	\LLambda^{\fm}(M)
	\simeq\Rlim_t K_\bullet(\bxt;M).
$
 Here completion with respect to
	$(\bx)$ and with respect to $\fm$ agree, because their powers are
	cofinal. The natural augmentations are compatible quasi-isomorphisms
$K_\bullet(\bxt;M)\to M/\mathfrak q_tM.$
	The quotient transition maps are surjective. Thus their higher inverse
	limit vanishes, and
$
	\Rlim_t M/\mathfrak q_tM
	\simeq\varprojlim_t M/\mathfrak q_tM
	=\widehat M.
$
	This proves $\Tor_i^R(k,M)\cong\Tor_i^R(k,\widehat M)$. The natural map to derived
	completion preserves Tor with $k$, by \cite[(2.1)]{IYENGAR}; hence
$\LLambda^{\fm}(M)\simeq\widehat M$ follows.
\end{proof}

\begin{proposition}\label{thm:equivalence}
	For a big Cohen--Macaulay $R$-module $M$, the following are equivalent:
	\begin{enumerate}
		\item $\Tor_n^R(k,M)=0$ for some $n\ge1$;
		\item $\Tor_i^R(k,M)=0$ for every $i\ge1$;
		\item $\widehat M$ is faithfully flat as an $R$-module.
	\end{enumerate}
	Each condition implies that $R$ is Cohen--Macaulay.
\end{proposition}

\begin{proof}
	The implication (1)$\Rightarrow$(2) and the last assertion follow from
 Corollary ~\ref{15new}. Suppose (2) holds. Lemma~\ref{lem:completion}
	gives $\Tor_1^R(k,\widehat M)=0$, and $\widehat M$ is a derived
	$\fm$-complete module. Apply \cite[Lemma~2.1]{IYENGAR} with
	$\mathfrak a=\fm$ and $n=1$. The only prime containing $\fm$ is
	$\fm$ itself, so that lemma gives $\fd_R\widehat M\le0$.
	The nonzero residue quotient in Lemma~\ref{lem:completion} makes this
	flat module faithfully flat. This proves (3). Conversely, (3) and
$\Tor_i^R(k,M)\cong\Tor_i^R(k,\widehat M) $ give (2), and (2)$\Rightarrow$(1) is immediate.
\end{proof}

\begin{definition}\label{mtor}
 $\fm$-torsion means that each element   is annihilated by
some power of $\fm$.
\end{definition}
\begin{corollary}\label{cor:torsion}
	Under the hypotheses of Corollary ~\ref{15new} 	and  $T$ be $\fm$-torsion. Then
$\Tor_+^R(T,M)=0.$
\end{corollary}

\begin{proof}
	For a finite-length module this follows from a composition series,
	the vanishing for $k$, and the long exact Tor sequence. Every
	$\fm$-torsion module is the filtered union of its finitely generated
	submodules, which have finite length. Now use that Tor commutes with
	filtered direct limits.
\end{proof}

\begin{corollary}\label{cor:pure}
	Let $S$ be a big Cohen--Macaulay $R$-algebra, and assume that
	$\Tor_n^R(k,S)=0$ for some $n\ge1$. Then $\widehat S$ is a faithfully
	flat $R$-algebra, and the structural map $R\to S$ is pure. In
	particular, $IS\cap R=I$ for every ideal $I$ of $R$.
\end{corollary}

\begin{proof}
Proposition~\ref{thm:equivalence} proves faithful flatness. For every
	$R$-module $L$, the composite
$L\to L\otimes_R S \to L\otimes_R\widehat S$
	is injective by faithful flatness. Its first map is therefore injective,
	which is precisely purity. Taking $L=R/I$ gives the ideal contraction
	statement. Purity for $L=R$ also shows that the structural map itself
	is injective.
\end{proof}

\begin{corollary}
	Under the hypotheses of  ~\ref{15new}, if $M$ is
	$\fm$-adically complete, then $M$ is faithfully flat. If $M$ is
	finitely generated, then $M$ is a nonzero finite free $R$-module.
\end{corollary}

\begin{proof}
	The complete case follows from Proposition~\ref{thm:equivalence}. In the
	finite case $\widehat M\cong\widehat R\otimes_R M$. Faithful flatness
	of $R\to\widehat R$ and descent of flatness show that $M$ is flat;
	a finite flat module over a local ring is free.
\end{proof}
\subsection{Failure of weak rigidity: some examples}

This subsection provides a detailed analysis of $\beta_1(H^1_{\mathfrak{m}}(R))$ for a specific non-Cohen--Macaulay ring, illustrating the failure of rigidity.

\begin{example}(Lichtenbaum, \cite{L})
Let $R=\frac{k[x,y]}{(x^2,xy)}$ and $M=R/\fp$, where $\fp:=(x)$ is a prime ideal. Since $\fp\in \Spec(R)\setminus \{\fm\}$  and $\dim(R)=1$, it follows that $M$ is a  Cohen--Macaulay module, as multiplication with any $r\in\fm\setminus \fp$ gives an injection $R/\fp \stackrel{r}\longrightarrow R/\fp$. In particular, $M$ is maximal Cohen--Macaulay.  Note that $x \in \sqrt{yR}$ because $x^2=0 \in yR$. Hence $\sqrt{yR}=(x,y)$, so $y$ is a system of parameters. It remains to recall item (3) from \cite[Page 226]{L} that $\operatorname{Tor}_1^R(M,R/yR)=0$, but $\operatorname{Tor}_2^R(M,R/yR)\neq 0$.
\end{example}

Let us check the non-vanishing property of $\beta_d$
for our favorite table $\{H^d_{\mathfrak{m}}(R),E(R/\mathfrak{m}), R^+\}$, with the convenience that
$R^+:=\oplus_{\fp\in\Ass(R)} (R/\fp)^+$.
\begin{example}
 Suppose $k=\bar{k}$, and let $R:=\frac{k[x,y]}{(x^2,xy)}$. The following holds:
\begin{itemize}
\item[(i)] 
One has 	$\beta_1(H^1_{\mathfrak{m}}(R))\neq 0$.
\item[(ii)] One has 
$\beta_1(E(R/\mathfrak{m}))\neq 0$.
\item[(iii)] One has 		$\beta_1(R^+)\neq 0$.
	\end{itemize}
\end{example}

\begin{proof}
(i)	First note that $R$ is not Cohen--Macaulay since $\operatorname{depth}(R)=0\neq 1=\dim(R)$. Indeed, $H^0_{\mathfrak{m}}(R)=xR\neq 0$. Consider the beginning of a free resolution of $k$:
	$$
(F_\bullet,d_\bullet):=	\cdots\lo R^3 \xrightarrow{\begin{pmatrix}x&y&0\\0&0&x\end{pmatrix}} R^2 \xrightarrow{(x\;\;y)} R \longrightarrow k \longrightarrow 0,
	$$which may be checked by hand that is exact.
Recall that $\sqrt{(y)}=\mathfrak{m}$, hence $y$ is a system of parameters. Therefore $H^1_{\mathfrak{m}}(R) = \frac{R_y}{\operatorname{im}(\rho)}$, where $\operatorname{im}(\rho)=\operatorname{im}(R\to R_y)$, see \cite[Corollary 2.2.21]{BSh}. Now $R_y = \{\frac{r}{y^n}\mid r\in R,\ n\ge 0\}$. Since $xy=0$ in $R$, we have $\frac{x}{y}=0$. Hence
$$R_y = \left\{\frac{1+a_1y+\cdots+a_ny^n}{y^m}\mid a_i\in k,\ m,n\ge 0\right\}.
	$$	
We tensor  $F_\bullet$ with $H^1_{\mathfrak{m}}(R)$, and this  gives $  H^1_{\mathfrak{m}}(R) ^{\oplus 3} \xrightarrow{A}  H^1_{\mathfrak{m}}(R) ^{\oplus 2} \xrightarrow{(x\;\;y)} H^1_{\mathfrak{m}}(R)$. Therefore
$$
\operatorname{Tor}_1^R(H^1_{\mathfrak{m}}(R),k) = H_1[ H^1_{\mathfrak{m}}(R) ^{\oplus 3} \stackrel{\varepsilon_2}\lo  H^1_{\mathfrak{m}}(R) ^{\oplus 2} \stackrel{ \varepsilon_1}\lo H^1_{\mathfrak{m}}(R)].
$$
	
Let us compute this homology. First, $\ker(\varepsilon_1) = \{\binom{\alpha}{\beta}\mid \alpha x+\beta y=0\}$, where $\alpha,\beta\in H^1_{\mathfrak{m}}(R)$ ,  
and recall that  $$\im(\varepsilon_2) =\bigg\{ \begin{pmatrix}\alpha\\\beta\\\gamma\end{pmatrix}\begin{pmatrix}x&y&0\\0&0&x\end{pmatrix}\mid \alpha ,\beta,\gamma\in H^1_{\mathfrak{m}}(R)\bigg\}=\bigg\{ \begin{pmatrix}\alpha x+y\beta+0\gamma\\ 0\alpha+0\beta+x\gamma\end{pmatrix}\bigg\}.$$ 
Then
\begin{itemize}
\item[(a)] 
$\ker(\varepsilon_1) = \{\binom{\alpha}{\beta}\mid \beta y=0\}$, since
$xH^1_{\mathfrak{m}}=0$.
\item[(b)] $\operatorname{im}(\varepsilon_2) =\operatorname{im}(A) = \bigg\{ \begin{pmatrix} y\beta \\ 0\end{pmatrix}\mid\beta\in H^1_{\mathfrak{m}}(R)\bigg\}$, as $x\gamma=0$.
	\end{itemize}
 In particular,

\begin{itemize}
\item[$\bullet$]
$	\binom{0}{\frac1y+\operatorname{im}(R)} \stackrel{(a)}\in \ker(1_{H^1_{\mathfrak{m}}(R)}\otimes d_1),$ but
\item[$\bullet$]	
$		\binom{0}{\frac1y+\operatorname{im}(R)} \stackrel{(b)}\notin \operatorname{im}(1_{H^1_{\mathfrak{m}}(R)}\otimes d_2).
$ 
\end{itemize}

Therefore $\operatorname{Tor}_1^R(H^1_{\mathfrak{m}}(R),R/\mathfrak{m}) = H_1(\frac{R_y}{\operatorname{im}(R)}\otimes_R F_\bullet) \neq 0$.

(ii)	Let $F_\bullet$ be as (i), and set $A:=k[[X,Y]]$. Then there is a ring homomorphism $A\to R$.  By \cite[Lemma 3.16]{BH}, $0\neq  \operatorname{Hom}_A(R,E_A(k))$,   is injective and its associated prime is $\mathfrak{m}$. By Matlis decomposition theorem,
$\operatorname{Hom}_A(R,E(k))	=\oplus_t E_R(R/\mathfrak{m})$ for some $t>0$.
Then, we show $\beta_1(\operatorname{Hom}_A(R,E(k)))\neq 0$. Recall that $\operatorname{Tor}_1^R(R/\mathfrak{m},\operatorname{Hom}_A(R,E_A(k)))$ is isomorphic to $\operatorname{Hom}_A(\operatorname{Ext}_R^1(R/\mathfrak{m},A),E_A(k))$, see \cite[3.2.1]{en}. Thus it suffices to show $E_1:=\operatorname{Ext}_R^1(R/\mathfrak{m},A)\neq 0$.  We apply   $\operatorname{Hom}(-,A)$ to $F_\bullet$, and we obtain
$$A \xrightarrow{\binom{x}{y}} A^2 \xrightarrow{\begin{bmatrix}x&0\\y&0\\0&x\end{bmatrix}} A^3.
	$$
We denote the boundary of this complex  by $\partial_\bullet$.	For $(x,y)\in A^2$, we observe
$$
	\begin{bmatrix}x&0\\y&0\\0&x\end{bmatrix}\binom{x}{x} = \begin{pmatrix}x^2\\xy\\x^2\end{pmatrix}= \begin{pmatrix}0\\0\\0\end{pmatrix}=0.
	$$
Thus $\binom{x}{x}\in \ker(\partial_2)$, but $\binom{x}{x}\notin \operatorname{im}(A \xrightarrow{(x\;y)} A^2) = \{\binom{ax}{ay}\mid a\in A\}$. Hence the complex is not exact, so $\operatorname{Ext}_R^1(R/\mathfrak{m},A)\neq 0$. Therefore $\operatorname{Tor}_1^R(R/\mathfrak{m},E_R(k))\neq 0$, and consequently $\beta_1(E_R(R/\mathfrak{m}))\neq 0$.

(iii) Recall that $\Ass(R)=\{\fm, (x)\}$. Since $R^+=\oplus_{\fp\in\Ass(R)} (R/\fp)^+=k\oplus (k[[y]])^+$, and $\beta_1(k)=1$, we see that $\beta_1(R^+)=1+\beta_1((k[[y]])^+)\neq 0$, as desired claim. \end{proof}

\begin{proposition}\label{prop:non-rigid}
Let $R$ be a Cohen--Macaulay ring and $M$ a non-Cohen--Macaulay module of dimension $d_0< d-1:=\dim(R)-1$ and depth $t\le d-2$. Then $H^d_{\mathfrak{m}}(R)$ is not rigid.
\end{proposition}

\begin{proof}
Since $R$ is Cohen--Macaulay, the flat resolution of $H^d_{\mathfrak{m}}(R)$ is given by
$$0 \longrightarrow R \longrightarrow\bigoplus_{i=1}^d R_{x_i} \longrightarrow \cdots \longrightarrow R_{x_1.   \cdots . x_d} \longrightarrow H^d_{\mathfrak{m}}(R) \longrightarrow 0,
	$$
where $x_1,\ldots,x_d$ is a system of parameters. Tensoring with $M$ gives a complex whose homology is
$\operatorname{Tor}^{R}_{i}(M,H^{d}_{\mathfrak{m}}(R))=H^{d-i}_{\mathfrak{m}}(M).$
Suppose for contradiction that it is tor-rigid. Then $\operatorname{Tor}^{R}_{1}(M,H^d_{\mathfrak{m}}(R))=H^{d-1}_{\mathfrak{m}}(M)=0$ by Grothendieck's vanishing theorem (see \cite[6.1.2]{BSh}), as $d_0<d-1$. On the other hand, $\operatorname{Tor}^{R}_{d-t}(M,H^d_{\mathfrak{m}}(R))=H^{t}_{\mathfrak{m}}(M)\neq 0$, since $t=\operatorname{depth}(M)$ (see \cite[6.2.7]{BSh}). Therefore $H^d_{\mathfrak{m}}(R)$ is not rigid.
\end{proof}

\section{Consequences of (non-)vanishing of $\beta_d(M)$}	
\subsection{Non-vanishing of $\beta_d(M)$}
	
In this section, we prove {Theorem} \ref{13}, i.e., the non-vanishing of $\beta_d(M)$ for $\mathfrak{m}$-torsion modules over Cohen--Macaulay rings, and explore consequences for the injective hull $E_R(k)$.
Recall from
{Definition} \ref{mtor}
that an $R$-module $M$ is called \emph{$\mathfrak{m}$-torsion} if $\Gamma_{\mathfrak{m}}(M)=M$, i.e., for each $m\in M$ there exists $n>0$ such that $\mathfrak{m}^n m =0$.

\begin{remark}
Every module of finite length is $\mathfrak{m}$-torsion. Any direct limit of  {$\mathfrak{m}$-torsion} is again  {$\mathfrak{m}$-torsion}. In particular,
any artinian module is  {$\mathfrak{m}$-torsion}.
	The module $\oplus_{\mathbb{N}}E_R(k)$ is a non-artinian module which is  {$\mathfrak{m}$-torsion}. 
\end{remark}
	
\begin{theorem}\label{thm:top-betti}
Suppose $(R,\mathfrak{m})$ is a Cohen--Macaulay local ring of dimension $d$. Let $0\neq\Omega$ be an $\mathfrak{m}$-torsion module. Then $\beta_d(\Omega)\neq 0$.
\end{theorem}

\begin{proof}
Assume, toward a contradiction, that $\beta_d(\Omega)=\dim_k(\operatorname{Tor}_d^R(k,\Omega))=0$, i.e., $\operatorname{Tor}_d^R(k,\Omega)=0$. We first establish that $\operatorname{Tor}_d^R(L,\Omega)=0$ for every finite-length module $L$. This follows by induction on the length $\ell(L)$: the case $\ell(L)=1$ reduces to $L\cong k$, which holds by the assumption. For $\ell(L)>1$, choose a maximal proper submodule $L'\subsetneq L$, so $\ell(L')=\ell(L)-1$ and $L/L'\cong k$. The induced long exact sequence in Tor from $0\to L'\to L\to k\to 0$ gives $\operatorname{Tor}_d^R(L',\Omega)\to \operatorname{Tor}_d^R(L,\Omega)\to \operatorname{Tor}_d^R(k,\Omega)$, where the left and right terms vanish by the induction hypothesis and the initial assumption, respectively; hence $\operatorname{Tor}_d^R(L,\Omega)=0$. Now consider any artinian module $A$, expressed as a direct limit $A=\varinjlim A_i$ of its finitely generated submodules $A_i\subseteq A$. Since $A$ is artinian, each $A_i$ has finite length, so the preceding paragraph yields $\operatorname{Tor}_d^R(A_i,\Omega)=0$ for all $i$. As Tor commutes with direct limits, we obtain $\operatorname{Tor}_d^R(A,\Omega)=\varinjlim \operatorname{Tor}_d^R(A_i,\Omega)=0$. Since $R$ is Cohen--Macaulay, choose a system of parameters $\mathbf{x}=x_1,\dots,x_d$. Then the local cohomology module $H_{\mathfrak{m}}^d(R)\cong H_{\mathbf{x}}^d(R)$ is artinian, and the isomorphism $\Omega\cong \operatorname{Tor}_d^R(H_{\mathbf{x}}^d(R),\Omega)$ from \cite[Proposition 7.8]{moh4} applies. But because $H_{\mathfrak{m}}^d(R)$ is artinian, the vanishing result just proven gives $\operatorname{Tor}_d^R(H_{\mathfrak{m}}^d(R),\Omega)=0$, which forces $\Omega=0$, contradicting the hypothesis that $\Omega$ is non-zero. Therefore $\beta_d(\Omega)\neq 0$.
\end{proof}
	
\begin{remark}
Concerning the other Betti numbers $\beta_i(\Omega)$, one cannot expect a similar statement. The following well-known fact illustrates this phenomenon.
\end{remark}
	
\begin{fact}
Let $(R,\mathfrak{m})$ be a regular local ring and let $E=E_R(R/\mathfrak{m})$ denote the injective hull of the residue field. Then
$\beta_i(E)\neq 0$ if and only if $ i=\dim R.$ 
	\end{fact}
	
\begin{proof}
Since $R$ is regular, the injective dimension of $E$ is precisely $\dim R$. The minimal injective resolution of $E$ therefore has length $\dim R$. Applying Matlis duality to this resolution yields a minimal free resolution of the completion $\widehat{R}$ over itself. As $\widehat{R}$ is regular of dimension $d=\dim R$, its $d$-th Betti number is non-zero while all lower Betti numbers vanish; the same conclusion thus holds for $E$.
	\end{proof}
	
\begin{corollary}\label{46}
	Let $(R,\mathfrak{m})$ be a $d$-dimensional Cohen--Macaulay local ring and let $0 \neq M$ be an $R$-module with $\beta_d(M) = 0$. Then
	$$
	\beta_{d+1}(M/\Gamma_{\mathfrak{m}}(M)) \neq 0 \quad \text{if and only if} \quad \depth(M) > 0.
	$$
\end{corollary}

\begin{proof}
	Suppose first that $\depth(M) > 0$. By \cite[Theorem V.2]{sc3}, we have $\beta_{d+1}(M) = 0$. Since $\Gamma_{\mathfrak{m}}(M) = 0$, it follows that
	$
	\beta_{d+1}(M/\Gamma_{\mathfrak{m}}(M)) = \beta_{d+1}(M) = 0.
	$	Conversely, assume that $\depth(M) = 0$. Again by \cite[Theorem V.2]{sc3}, we have $\beta_{d+1}(M) = 0$. Consider the short exact sequence
$
	0\to\Gamma_{\mathfrak{m}}(M) \to M \to\overline{M} := M/\Gamma_{\mathfrak{m}}(M) \to 0.
$
	Applying $\operatorname{Tor}_*^R(-,k)$ gives the induced connecting homomorphism
	$$
	0 = \operatorname{Tor}_{d+1}^R(M,k) \longrightarrow \operatorname{Tor}_{d+1}^R(\overline{M},k) \longrightarrow \operatorname{Tor}_d^R(\Gamma_{\mathfrak{m}}(M),k) \longrightarrow \operatorname{Tor}_d^R(M,k) = 0,
	$$
	yielding the equality
\[
	\beta_{d+1}(\overline{M}) = \beta_d(\Gamma_{\mathfrak{m}}(M)). \tag{$\ast$}
\]
	Since $\depth(M) = 0$, we have $0 \neq \Gamma_{\mathfrak{m}}(M)$, which is $\mathfrak{m}$-torsion. It follows immediately from Theorem~\ref{thm:top-betti} that $\beta_d(\Gamma_{\mathfrak{m}}(M)) \neq 0$. In view of $(\ast)$, we conclude $\beta_{d+1}(M/\Gamma_{\mathfrak{m}}(M)) \neq 0$.
\end{proof}

Over regular rings, Corollary~\ref{46} can be polished further:

\begin{proposition}\label{71}
	Let $(R,\mathfrak{m})$ be a $d$-dimensional regular local ring and let $M$ be an $R$-module. Then
	$
	\beta_d(M) \neq 0$   {if and only if} $\Gamma_{\mathfrak{m}}(M) \neq 0.
	$
\end{proposition}

\begin{proof}
	$(\Leftarrow)$ Suppose $\Gamma_{\mathfrak{m}}(M) \neq 0$. Then there exists a non-zero element $x \in M$ annihilated by $\mathfrak{m}$, yielding an injection $0 \to R/\mathfrak{m} \xrightarrow{\varphi} M$. Let $C$ denote the cokernel, so we have a short exact sequence $0 \to k \to M \to C \to 0$. Since $\operatorname{gldim}(R) = d$, we have $\operatorname{Tor}_{d+1}^R(k,C) = 0$. Applying $\operatorname{Tor}_*^R(k,-)$ gives the exact segment
$
	0 = \operatorname{Tor}_{d+1}^R(k,C) \to \operatorname{Tor}_d^R(k,k) \to \operatorname{Tor}_d^R(k,M).
$
	Because $\operatorname{Tor}_d^R(k,k) \cong k \neq 0$, it follows that $\operatorname{Tor}_d^R(k,M) \neq 0$, and hence $\beta_d(M) \neq 0$.
	
	$(\Rightarrow)$ Conversely, suppose $\Gamma_{\mathfrak{m}}(M) = 0$. Write $M = \varinjlim_{i \in I} M_i$ as a direct limit of its finitely generated submodules. Since each $M_i \subseteq M$, we have $\Gamma_{\mathfrak{m}}(M_i) \subseteq \Gamma_{\mathfrak{m}}(M) = 0$, so $\depth(M_i) > 0$ for every $i$. By the Auslander--Buchsbaum formula applied over the regular local ring $R$, each $M_i$ has projective dimension strictly less than $d$; hence $\operatorname{Tor}_d^R(k,M_i) = 0$ for all $i$. Taking direct limits and using the commutation of Tor with direct limits yields $\operatorname{Tor}_d^R(k,M) = 0$, so $\beta_d(M) = 0$.
\end{proof}
\subsection{Connections to local cohomology}
This section contains further results on Matlis reflexive modules and the relationship between Betti numbers and local cohomology. By $(-)^\vee$ we mean the Matlis dual.

\begin{remark}
	Let $M$ be an $R$-module, not necessarily finitely generated. The dimension of $M$ is the supremum of the lengths of chains of prime ideals in the support of $M$. In particular, $\dim(M) = \dim(\operatorname{Supp}(M))$; see \cite[6.1.1 Reminder]{BSh} for a connection to Grothendieck's vanishing theorem.
\end{remark}

\begin{proposition}\label{72}
	Let $(R,\mathfrak{m})$ be a $d$-dimensional local ring and let $0 \neq M$ be an $R$-module. The following hold:
	\begin{itemize}
		\item[(i)] If $\beta_d(M) = 0$, then $H^d_{\mathfrak{m}}(M^\vee) = 0$.
		\item[(ii)] If $\mu_d(M) = 0$, then $H^d_{\mathfrak{m}}(M) = 0$. In particular, $\dim(M) < d$.
	\end{itemize}
\end{proposition}

\begin{proof}
	(i) The condition $\beta_d(M) = 0$ means $\operatorname{Tor}_d^R(k,M) = 0$. Using the isomorphism $\operatorname{Tor}_i^R(k,M)^\vee \cong \operatorname{Ext}_R^i(k,M^\vee)$, we obtain $\operatorname{Ext}_R^d(k,M^\vee) = 0$. Since
$
	H^d_{\mathfrak{m}}(M^\vee) \cong \varinjlim_n \operatorname{Ext}_R^d(R/\mathfrak{m}^n, M^\vee),
$
	and each $\operatorname{Ext}_R^d(R/\mathfrak{m}^n, M^\vee)$ is a subquotient of a direct sum of copies of $\operatorname{Ext}_R^d(k, M^\vee) = 0$ (for more details, see the argument of Theorem~\ref{thm:top-betti}), it follows that $H^d_{\mathfrak{m}}(M^\vee) = 0$, as claimed.
	
	(ii) The condition $\mu_d(M) = 0$ means $\operatorname{Ext}_R^d(k, M) = 0$. By the same argument as in part (i), we obtain $\operatorname{Ext}_R^d(R/\mathfrak{m}^n, M) = 0$ for each $n$. This implies that $H^d_{\mathfrak{m}}(M) = 0$. The last assertion follows from Grothendieck's vanishing theorem (see \cite[6.1.2]{BSh}).
\end{proof}

\begin{remark}
	\begin{itemize}
		\item[(i)] Let $R$ be a regular local ring of dimension $d$ and let $M$ be $\mathfrak{m}$-torsion-free. Then $H^d_{\mathfrak{m}}(M^\vee) = 0$. Indeed, since $M$ is $\mathfrak{m}$-torsion-free, we have $\Gamma_{\mathfrak{m}}(M) = 0$. By Proposition~\ref{71}, this implies $\beta_d(M) = 0$. The result then follows immediately from Proposition~\ref{72}.
		
		\item[(ii)] If $R$ is $d$-dimensional, then $\beta_d(E_R(k)) \neq 0$. Indeed, suppose to the contrary that $\beta_d(E_R(k)) = 0$. Then $\operatorname{Tor}_d^R(E_R(k), k) = 0$. By Proposition~\ref{72}(i), $H^d_{\mathfrak{m}}(E_R(k)^\vee) = 0$. But $E_R(k)^\vee \cong \widehat{R}$, and therefore $H^d_{\mathfrak{m}}(\widehat{R}) \neq 0$, a contradiction.
		
		\item[(iii)] Suppose $R$ is Cohen--Macaulay. Then $\beta_i(E_R(k)) = 0$ if and only if $i \neq \dim(R)$. Indeed, as before, note that $E_R(k)^\vee \simeq \widehat{R}$ and $\widehat{R}$ is Cohen--Macaulay, and $H^i_{\mathfrak{m}}(\widehat{R}) \neq 0$ exactly when $i = \dim(R)$.
		
		\item[(iv)] Is the Cohen--Macaulay hypothesis in (iii) really needed? Yes: if the equivalence holds, then $\operatorname{depth}(R) = \dim(R)$.
	\end{itemize}
\end{remark}

\begin{proposition}
	Let $R$ be a complete local ring and let $M$ be a nonzero Matlis reflexive module of dimension $r \neq 1$. Then $H^r_{\mathfrak{m}}(M) \neq 0$. In particular, $\mu_r(M) \neq 0$.
\end{proposition}

\begin{proof}
	If $r = 0$, then $M$ is artinian, whence $H^0_{\mathfrak{m}}(M) = M \neq 0$. Now assume $r > 0$. Since $R$ is complete, every Matlis reflexive module sits in a short exact sequence
$
	0 \to F \to M \to A \to 0,
$
	where $F$ is finitely generated and $A$ is artinian (see \cite[page 92, Ex.~6]{en}). Recall that $F \neq 0$, as $r > 0 = \dim(A)$. Consequently, $\dim(M) = \dim(F) = r$. Applying the long exact sequence in local cohomology to the above sequence yields
	$
	H^{r-1}_{\mathfrak{m}}(A) \to H^r_{\mathfrak{m}}(F) \to H^r_{\mathfrak{m}}(M).
	$
	One has $H^{r-1}_{\mathfrak{m}}(A) = 0$ by Grothendieck's vanishing theorem because $\dim(A) = 0 < r-1$ (here we used the assumption $r \neq 1$). Thus, we obtain an injection $H^r_{\mathfrak{m}}(F) \hookrightarrow H^r_{\mathfrak{m}}(M)$. But by Grothendieck's non-vanishing theorem, we have $H^r_{\mathfrak{m}}(F) \neq 0$ (see \cite[6.1.4]{BSh}), so we conclude that $H^r_{\mathfrak{m}}(M) \neq 0$. The last assertion follows from Proposition~\ref{72}(ii).
\end{proof}


\begin{example}
	Let $(R,\mathfrak{m})$ be a $1$-dimensional complete local integral domain with field of fractions $Q$. Recall that $0 \in \operatorname{Supp}(Q)$. Then $\dim(Q) = 1$, and note that $Q = R_x$ for any $0 \neq x \in \mathfrak{m}$. By \cite[Remark~2.2.20]{BSh}, the short exact sequence
$ 0 \to R \to Q \to Q/R \to 0 $
	exhibits $Q/R \cong H^1_{\mathfrak{m}}(R)$ as an artinian module, so $Q$ is Matlis reflexive. However, $Q$ is injective, whence $H^1_{\mathfrak{m}}(Q) = 0$. This shows that the assumption $r \neq 1$ in the proposition is necessary. Moreover, since $Q$ is injective, we have $\mu_1(Q) = 0$ and $\operatorname{Ext}^1_R(R/\mathfrak{m}, Q) = 0$.
\end{example}

\begin{remark}
	A natural question arises: how can one construct a Matlis reflexive module that is not a direct sum of a finitely generated module and an artinian module? Here, we present such a module. Consider any $1$-dimensional complete local ring $(R,\mathfrak{m})$. In this setting, by using a result of Matlis, we know $\operatorname{Ext}^1_R(E_R(k), R)$ is non-zero, so there exists a non-split short exact sequence
$
	0 \to R \to K \to E_R(k) \to 0.
$
	The middle module $K$ is then Matlis reflexive, yet it is neither artinian nor noetherian, and in particular it cannot be decomposed as a direct sum of a finitely generated module and an artinian module.
\end{remark}

\section{Infinite  integral extensions: Bass and Betti numbers}
	
This section is devoted to the absolute integral closure and contains our main results on both Tor and Ext for $R^+$.
	
\subsection{Tor-vanishing for $R^+$}
We start with the following auxiliary result.

\begin{proposition}
	Suppose $R$ is Cohen--Macaulay and $M$ is a module with $\depth(M) = \dim(R) - 1$. Then $\beta_1(M) \neq 0$.
\end{proposition}

\begin{proof}
	Assume to the contrary that $\beta_1(M) = 0$, i.e., $\operatorname{Tor}_1^R(M, k) = 0$. Let $\underline{x} = x_1, \dots, x_{d-1}$ be both an $M$-sequence and an $R$-sequence, which exists because $\depth(M) = d - 1$ and $R$ is Cohen--Macaulay. As observed in \cite[Lemma 2.1]{sc1}, we have
$
	\operatorname{Tor}_1^{\overline{R}}\!\left(\frac{M}{\underline{x}M}, k\right) \cong \operatorname{Tor}_1^R(M, k) = 0,
$
	where $\overline{R} = R/(\underline{x})$. By a standard argument, the vanishing of $\operatorname{Tor}_1^{\overline{R}}(\overline{M}, k)$ implies that
	$
	\operatorname{Tor}_1^{\overline{R}}(\overline{M}, L) = 0
	$
	for every finite-length $\overline{R}$-module $L$. In particular, choosing $y$ to be a system of parameters for the one-dimensional Cohen--Macaulay ring $\overline{R}$, we get
	$$
	0 = \operatorname{Tor}_1^{\overline{R}}\!\left(\overline{M}, \frac{\overline{R}}{y\overline{R}}\right) = \ker\left(\overline{M} \xrightarrow{y} \overline{M}\right).
	$$
	This forces multiplication by $y$ on $\overline{M}$ to be injective. It follows that $\depth(\overline{M}) = 1$. Consequently,
$
	\depth(M) = \depth_{\overline{R}}(\overline{M}) + (d - 1) = d,
$
	contradicting the hypothesis that $\depth(M) = d - 1$. Therefore $\beta_1(M) \neq 0$.
\end{proof}

\begin{corollary}
Suppose $R$ is a $d$-dimensional Cohen--Macaulay local ring containing $\mathbb{Q}$. If $d>2$, then $\Tor_1^R(R^+,k)\neq 0$.
	\end{corollary}
	
\begin{proof}
For $d=3$, the result follows from the preceding proposition. Indeed, $R^+$ is a direct limit of normal rings, each of which satisfies Serre's condition $(S_2)$, so $\depth(R^+)\ge 2$. Since $R^+$ is not big Cohen--Macaulay  (see \cite[Ex. 7.2]{h}), we must have $\depth(R^+)=\dim(R)-1=2$. The proposition then yields $\beta_1(R^+)\neq 0$. For $d>3$, the result is established in \cite{p}.
\end{proof}

\begin{example}
	Suppose $R$ is a $2$-dimensional regular local ring. Then $\operatorname{Tor}_+^R(R^+, k) = 0$. Indeed, this is easy to see since $\depth_R(R^+) = 2$. Recall from \cite[4.1]{BS} that
$0 = \depth(R) - \depth(R^+) = \operatorname{resdim}(R^+).$
	Therefore $\operatorname{Tor}_i^R(R^+, k) = 0$ for all $i > 0$. Alternatively, under the additional completeness assumption, one may use \cite[23.1]{mat} and deduce that $R \to R^+$ is flat, since a direct limit of flat modules is flat, and $R^+$ is a direct limit of local Cohen--Macaulay rings. Since $\pd(k)<\infty$ it follows that $\operatorname{Tor}_+^R(R^+, k) = 0$ (see {Observation} \ref{3.3}). 
\end{example}

\begin{corollary}\label{qu}
	Let $R$ be a quotient singularity containing $\mathbb{Q}$. If
	$
	\operatorname{Tor}_i^R(R^+, R/\mathfrak{m}) = 0
	$
	for some $i > 0$, then $R$ is regular.
\end{corollary}

\begin{proof}
	By definition, a quotient singularity is of the form $R = k[x_1, \dots, x_n]^G$ for a finite group $G$ acting linearly on $A := k[x_1, \dots, x_n]$. Recall that $R \subseteq A$ is an integral extension, and so $R$ is normal. Since $R$ is normal, every finite extension of $R$ splits (see \cite[Lemma 2.1]{p}). Hence the inclusion $A \hookrightarrow R^+$ splits, so $\operatorname{Tor}_i^R(A, k)$ is a direct summand of $\operatorname{Tor}_i^R(R^+, k) = 0$. Thus $\operatorname{Tor}_i^R(A, k) = 0$ as well.
Since $A$ is finitely generated as an $R$-module, it has finite flat dimension over $R$. By the Gruson--Raynaud theorem, $A$ has finite projective dimension over $R$. As $A$ is Cohen--Macaulay, the Auslander--Buchsbaum formula implies that $A$ is flat over $R$. Since $A$ is regular, it follows from \cite[2.2.12(a)]{BH} that $R$ is regular.
\end{proof}

\begin{discussion}(Reduction to complete rings).
Let $(R,\mathfrak m)$ be an excellent local domain. There are three natural
objects associated to $R$, namely
$\{
(\widehat{R})^+,
\widehat{R}\otimes_R R^+,
\widehat{R^+}\}.
$
It may be worth to finding conditions for which there is an
injective, surjective, or isomorphisms between these.
For example, let $(R,\mathfrak m)$ be an excellent analytically irreducible local domain.
Assume that
$
\widehat R\otimes_R R^+
$
is a domain. Then there exists a natural injective homomorphism
$
\widehat R\otimes_R R^+
\hookrightarrow
(\widehat R)^+.
$ Indeed, let $K:=\operatorname{Frac}(R)$ and fix an algebraic closure $\overline K$ of
$K$. Since
$
R^+=\varinjlim_{L/K} \overline R_L,
$
where $\overline R_L$ denotes the integral closure of $R$ in a finite field
$
\widehat R\otimes_R R^+
\cong
\varinjlim_{L/K}
\bigl(\widehat R\otimes_R \overline R_L\bigr).
$ Since $R$ is excellent, each ring
$
\widehat R\otimes_R \overline R_L
$
is normal. Moreover, it is finite over $\widehat R$, hence integral over
$\widehat R$. Therefore
$
\widehat R\otimes_R \overline R_L
$
is contained in the integral closure of $\widehat R$ inside an algebraic
closure of $\operatorname{Frac}(\widehat R)$, namely $(\widehat R)^+$.
Consequently,
$
\widehat R\otimes_R \overline R_L
\subseteq
(\widehat R)^+
$
for every finite extension $L/K$. Passing to the direct limit yields
$
\widehat R\otimes_R R^+
=
\varinjlim_{L/K}
\bigl(\widehat R\otimes_R \overline R_L\bigr)
\subseteq
(\widehat R)^+.
$	Hence there is a natural injective homomorphism
$
\widehat R\otimes_R R^+
\hookrightarrow
(\widehat R)^+.
$
\end{discussion}
The following completes the proof of Theorem~\ref{17}.

\begin{proposition}\label{thm:finite-CM-type}
	Let $R$ be a homogeneous Cohen--Macaulay domain of dimension $d$ over an algebraically closed field of characteristic $0$. Assume that $R$ is of finite Cohen--Macaulay type. If
	$
	\operatorname{Tor}_i^R(R^+, k) = 0
	$
	for some $i > d$, then $R$ is a UFD and has multiplicity at most two.
\end{proposition}

\begin{proof}
	By the classification of Eisenbud--Herzog (see \cite[Theorem 16.6]{lw}), the ring $R$ is isomorphic to one of the following:
	
\begin{itemize} \item[(i)] $k[x_0, \ldots, x_n]$ $\exists n \geqslant 0$;\quad (ii) $k[x_0, \ldots, x_n]/(x_0^2 + \cdots + x_n^2)$ $\exists n \geqslant 0$; \item[(iii)] $k[x]/(x^m)$ $\exists m \geqslant 1$; \ \ \quad(iv) $k[x, y]/(xy(x + y))$; \quad(v) $k[x, y, z]/(xy, yz, xz)$; \item[(vi)] $k[x_0, \ldots, x_m]/I_2\!\begin{pmatrix} x_0 & \cdots & x_{m-1} \\ x_1 & \cdots & x_m \end{pmatrix}$ for some $m \geqslant 1$; \item[(vii)] $k[x, y, z, u, v]/I_2\!\begin{pmatrix} x & y & u \\ y & z & v \end{pmatrix}$; and (viii) $k[x, y, z, u, v, w]/I_2 \begin{pmatrix} x & y & z \\ y & u & v \\ z & v & w \end{pmatrix}$. \end{itemize}
	
Since $R$ is a domain, cases (iii)--(v) cannot occur. Case (i) is regular, and hence $R$ is a UFD of multiplicity one.
By \cite[Example 16.4]{lw}, the ring in (viii) is a quotient singularity. It is well known that $R$ is normal, non-Gorenstein, and of minimal multiplicity (see \cite[Theorem 16.7]{lw}). Therefore Corollary~\ref{qu} shows that
	$
	\operatorname{Tor}_j^R(R^+, k) \neq 0
	$
for all $j > 0$. In particular, this ring is excluded by the assumption.
	Also, the ring in (vii) (see \cite[Example 16.2]{lw}) is isomorphic to
	$$
	R \cong k[x, y, z, u, v]/(yv - zu,\; yu - xv,\; xz - y^2).
	$$
	This is a quotient of a polynomial ring by a binomial prime ideal, and hence $R$ is toric. It is well known that $R$ is normal (for example, it is $(S_2)$ and $(R_1)$). Thus, we are in the situation of \cite[Theorem 1.1]{MP} to exclude this ring as well.
Moreover, by \cite[p.~291]{lw}, the ring in (vi) is isomorphic to
	$
	k[u^m, u^{m-1}v, \ldots, v^m],
	$
	and hence is also a quotient singularity. Therefore Corollary~\ref{qu} shows that
	$
	\operatorname{Tor}_j^R(R^+, k) \neq 0
	$
	for all $j > 0$. Thus these cases are excluded.
It remains to consider case (ii), namely the $A_1$-singularity
	$
	R = k[x_0, \ldots, x_n]/(x_0^2 + \cdots + x_n^2).
	$
	When $d = 2$, the result follows from Corollary~\ref{qu}; indeed, by a theorem of Esnault--Herzog, such a ring is a quotient singularity. Alternatively, one may use \cite[Theorem (A)]{p}.
Assume now that $d \geqslant 3$. It is well known that
	$
	\operatorname{Cl}(R) = 0
	$
	whenever $d \geqslant 4$; see \cite[Section 11]{f}. Hence $R$ is a UFD.
The remaining case is $d = 3$. In this case
	$
	R \cong k[x, y, z, w]/(xw - yz),
	$
	and
	$
	\operatorname{Cl}(R) \cong \mathbb Z
	$
	by \cite[11.3]{f}. Thus $R$ is not a UFD.
	Let $\mathfrak p = (x, y)$, which generates $\operatorname{Cl}(R)$. Following \cite[Example 2]{Rob}, one can construct a module-finite normal domain extension
	$
	A \cong R \oplus \mathfrak p^{(-2)}\eta
	$
	contained in $R^+$, where $\eta \in R^+$ is an element without a perfect square in $R$ (we may impose further restrictions on it; see \cite[Example 2]{Rob}).
	On the other hand, as observed in \cite{p}, the assumption
	$
	\operatorname{Tor}_i^R(R^+, k) = 0
	$
	implies
	$
	\operatorname{Tor}_i^R(A, k) = 0.
	$
	Now, recall that $\mathfrak p^{(-2)}\eta$ is a direct summand of $A$. Thus,
	$
	\operatorname{Tor}_i^R(\mathfrak p^{(-2)}\eta, k) = 0.
	$
	This shows that $\mathfrak p^{(-2)}\eta$ has finite projective dimension, as it is finitely generated as an $R$-module. Recall that $\mathfrak p^{(-2)}\eta \cong \mathfrak p^{(-2)}$ as an $R$-module.
	Set $M := \mathfrak p^{(-2)} = (\mathfrak p^{(2)})^\ast$. It is easy to see the following 3 properties:
(a)  $\operatorname{pd}_R(M) < \infty$;
(b)   $\operatorname{End}_R(M) = \overline{R} = R$ is a projective $R$-module;
 (c)  $M$ is reflexive.
Then $M$ is a locally Gorenstein module (see, e.g., \cite[Fact 3.5]{ufd}). In particular, it is Cohen--Macaulay, but then $A$ is Cohen--Macaulay, which it is not (see \cite[Example 2]{Rob}).
	Therefore the three-dimensional quadric hypersurface cannot occur.\footnote{This ring is toric, so one could argue as in \cite{MP}, though that paper postdates our work. We prefer the following self-contained argument.} Hence $R$ is a UFD.
	To complete the proof, recall that
	$
	e(k[\underline{x}]/(f)) = \deg(f),
	$
	where $f$ is homogeneous.
\end{proof}

 \begin{observation}
 	\begin{itemize}
 		\item[(i)] Let $A = \frac{k[x,y,z]}{(f_1,f_2)}$ be a $2$-dimensional ring. Assume that $\operatorname{Tor}_i^A(A^{+}, k) = 0$ for some $i > 0$. Then $(f_1,f_2) = (f)$ for some polynomial $f$.
 		
 		\item[(ii)] Suppose that $f$ is a binomial, $f = x^{a_1}y^{b_1}z^{c_1} - x^{a_2}y^{b_2}z^{c_2}$. In this case $R = k[x,y,z]/(f)$ is a toric ring. For simplicity, consider $f = x^3 - zy$. Define a $k$-algebra homomorphism $\pi: k[x,y,z] \longrightarrow k[u_1,u_2]$ by
 		$$
 		x \longmapsto u_1u_2, \qquad y \longmapsto u_1^2u_2, \qquad z \longmapsto u_1u_2^2.
 		$$
 		Then
 		$
 		\pi(x^3 - zy) = (u_1u_2)^3 - (u_1u_2^2)(u_1^2u_2) = 0.
 		$
 		Hence $x^3 - zy \in \ker(\pi)$. Since $\ker(\pi)$ is a height one prime ideal of the UFD $k[x,y,z]$, it is principal. Therefore $\ker(\pi) = (x^3 - zy)$. Consequently,
 		$
 		R = \frac{k[x,y,z]}{(x^3 - zy)} \hookrightarrow k[u_1,u_2].
 		$
 		Since $k[u_1,u_2]$ is normal, we obtain an inclusion $R \subseteq k[u_1,u_2] \subseteq R^{+}$. Now assume that $\operatorname{Tor}_i^R(k, R^{+}) = 0$ for some $i > 0$. Because $k[u_1,u_2]$ is a direct summand of $R^{+}$ (by the usual splitting argument), it follows that $\operatorname{Tor}_i^R(k, k[u_1,u_2]) = 0$. Hence $k[u_1,u_2]$ has finite flat dimension as an $R$-module. Since $k[u_1,u_2]$ is Cohen--Macaulay, the Auslander--Buchsbaum formula implies that $R$ is regular. But $R = \frac{k[x,y,z]}{(x^3 - zy)}$ is not regular, a contradiction. Therefore $\operatorname{Tor}_i^R(k, R^{+}) \neq 0$.
 		
 		\item[(iii)] More generally, suppose $H \subseteq \mathbb{N}^n$ is an affine semigroup. Define the semigroup ring
 		$$
 		R := \mathbb{Q}[H] := \bigoplus_{(a_1,\ldots,a_n) \in H} \mathbb{Q}\, x_1^{a_1}\cdots x_n^{a_n} \subseteq \mathbb{Q}[x_1,\ldots,x_n].
 		$$
 		Suppose $\operatorname{Tor}_i^R(k, R^+) = 0$ for some $i > 0$. Then $R$ is normal, which implies that $H$ is a normal semigroup. Recall that a normal semigroup satisfies $H = H_1 \cap \mathbb{Z}_+^t$, where $H_1$ is positive provided. Since, by assumption, $H \subseteq \mathbb{N}^n$, it follows that $H$ is positive. Hence $H$ is normal and positive. Such semigroups are called \emph{full}. For every full semigroup, we know that $k[H]$ is a direct summand of a polynomial ring (for more details, see \cite[Section 6.1]{BH}). The forthcoming work \cite{MP} shows that, under the assumption $\operatorname{Tor}_i^R(k, R^+) = 0$, $k[H]$ is indeed a polynomial ring.
 	\end{itemize}
 \end{observation}

\subsection{Ext-vanishing for $\{R^+,R^\infty\}$}

This subsection contains our main Ext-vanishing results for the absolute integral closure.
Here is the reason that explains why we study the condition $\mu_{j}(R^+)=0$ only for $j>\dim(R)$.
The results here are motivated from \cite{moh,moh5}.
\begin{example}
	Let $R$ be any two dimensional local ring. Then $\mu_2(R^+)\neq 0$.
\end{example}

\begin{proof}
In this case $R^+$ is direct limit of normal rings, and so it  big Cohen-Macaulay. This shows Ext-grade of $R^+$ is two, and so $\mu_2(R^+)\neq 0$.
\end{proof}

\begin{proposition}\label{811}
	Let $(R,\mathfrak{m})$ be a local ring of prime characteristic with isolated singularity. If
	$
	\operatorname{Ext}_{R}^{i}(R/\mathfrak{m}, R^+) = 0
	$
	for some $i > \dim(R)$, then $R$ is regular.
\end{proposition}

\begin{proof}
	The result of \cite[1.2]{IYENGAR} applied to $\operatorname{Ext}_R^i(R/\mathfrak{m}, R^+) = 0$ yields that
	$
	\operatorname{Ext}_R^j(R/\mathfrak{m}, R^+) = 0
	$
	for all $j \ge i$. This enables us to use \cite[Theorem VI.9]{sc3} and deduce that $R^+$ has finite injective dimension. By \cite[Example 3.17(i)]{ryo}, the ring $R$ is regular, as claimed.
\end{proof}
	
\begin{remark}
	One has $\operatorname{Ext}_{R}^{d}(R/\mathfrak{m}, R^{+}) \neq 0$ for complete $R$ under the above hypotheses, assuming $R$ is of mixed characteristic with $p \notin \mathfrak{m}^2$. Suppose not. By the Cohen structure theorem, $R$ is module-finite over a regular ring $A$. 
	Fix an algebraic closure L of $\Frac(R)$. Since $\Frac(R)$ is algebraic
	over $\Frac(A)$, the field $L$ is also an algebraic closure of
	$\Frac(R)$. Thus, upon taking both absolute integral closures inside $L$,
	we have
	$
	A^+=R^+.
	$
	 Then, by Proposition~\ref{72}, we would have $H_{\mathfrak{m}}^{d}(R^+) = 0$. However, a theorem of Hochster \cite[Theorem 6.1(5)]{h} states that $H_{\mathfrak{m}}^{d}(R^+) = H_{\mathfrak{m}}^{d}(A^+) \neq 0$, a contradiction.
\end{remark}

\begin{proposition}
Let $R$ be an $\mathbb{N}$-graded normal domain of dimension $2$, finitely generated over an equicharacteristic zero field $k$. If $\Ext_R^i(k,R^+)=0$ for some $i \geq 3$, then $R$ is regular.
	\end{proposition}
	
	\begin{proof}	
Normal rings of dimension two are Cohen-Macaulay, and satisfies Serre condition $(R_1)$. Since $\dim(R)=2$, it is of isolated singularity. Now, by \cite[Page 226]{sc3}, $R^+$ has finite injective dimension. Since $\mathbb{Q}\subseteq R$, and in view of \cite[Lemma 2.1]{p}, we know $R\to R^+$ splits. Thanks to our vanishing assumption, it follows that $\operatorname{Ext}^{i}_{R}(R/\mathfrak{m},R)=0$, so $R$ is Gorenstein, because by \cite{Ro2} Bass numbers are only nonzero in the range $[\depth(R), \id(R)]$, i.e., $\id(R)<\infty$. This implies that $R^+$ has finite flat dimension. Finally, we use  \cite[Theorem A]{p} to deduce that 
$R$ is regular.
	\end{proof}

\begin{proposition}
	Let $R$ be a local ring of dimension $d$ with an isolated quotient singularity containing $\mathbb{Q}$. If
$
	\operatorname{Ext}_R^i(k, R^+) = 0
$
	for some $i > d$, then $R$ is regular.
\end{proposition}

\begin{proof}
	Recall that quotient singularities in characteristic zero are Cohen--Macaulay (see, e.g., \cite[6.5.1]{BH}). This, together with $\operatorname{Ext}_R^i(k, R^+) = 0$, enables us to use \cite[page 226]{sc3} and deduce that $R^+$ has finite injective dimension.
	Also, quotient singularities in characteristic zero are normal domains (see, e.g., \cite[6.4.1]{BH}). Since $\mathbb{Q} \subseteq R$, and in view of \cite[Lemma 2.1]{p}, we deduce from normality that $R \to R^+$ splits. Our vanishing assumption then implies $\operatorname{Ext}_R^i(R/\mathfrak{m}, R) = 0$; hence $R$ is Gorenstein by \cite{Ro2}.
Now, using this property, we observe that $R^+$ has finite flat dimension, so
	$
	\operatorname{Tor}_j^R(R/\mathfrak{m}, R^+) = 0
	$
	for all $j \gg 0$. Using the quotient singularity property, along with Corollary~\ref{qu}, we conclude that $R$ is regular.
\end{proof}

\begin{proposition}
	Suppose $\mathbb{Q} \subseteq R$ has an isolated singularity, satisfies $(S_2)$, and has $\dim R \ge 4$. If $R$ is an almost complete intersection and
	$
	\operatorname{Ext}_R^i(k, R^+) = 0
	$
	for some $i > d$, then $R$ is a UFD.
\end{proposition}

\begin{proof}
	It follows that $R$ is normal, and since $R$ contains $\mathbb{Q}$, it is a splinter. Then the vanishing $\operatorname{Ext}_R^i(k, R^+) = 0$ implies that $R$ is Gorenstein. A theorem of Kunz \cite{kunz} (see also \cite[5.13]{moh5} for an alternative argument) states that any Gorenstein almost complete intersection is actually a complete intersection. Thus $R$ is a complete intersection.
	Since $R$ has an isolated singularity and $\dim R \ge 4$, for every non-maximal prime $\mathfrak{p}$ of height at most $3$, the localization $R_{\mathfrak{p}}$ is regular (hence a UFD). Therefore the Grothendieck criterion applies (see \cite{sga2}), and we conclude that $R$ is a UFD.
\end{proof}

\begin{fact}(See Rotman \cite{Rot}).
A submodule $N \subset M$ is pure iff for any finite free chain complex $F_\bullet$, the map 
$\alpha : \operatorname{Hom}_R(F_\bullet, N) \to \operatorname{Hom}_R(F_\bullet, M)$
sends coboundaries to coboundaries; equivalently, $H^i(\alpha)$ is injective for all $i$. 
Taking $F_\bullet$ to be a finite free resolution of $R/\mathfrak{m}$, purity gives the short exact sequence
$	0 \to \operatorname{Ext}_R^i(R/\mathfrak{m}, N) \to \operatorname{Ext}_R^i(R/\mathfrak{m}, M) \to \operatorname{Ext}_R^i(R/\mathfrak{m}, M/N) \to 0.$
\end{fact}
Perhaps the initial implication of $\beta_1(R^+) = 0$ was the purity property, discovered by Schoutens \cite[Theorem 2.2]{sc1}. Let us assume its variant in the prime characteristic case. We denote the Frobenius map $F:R\to R$, sending $x$ to $x^p$. This is indeed a morphism, as $\Char(R)=p>0$.
To this end, by $F(R)$, we mean $R$ as a group equipped with left and right scalar multiplication from $R$ given by
$
a \cdot r \star b = ab^p r,
$
where $a, b \in R$ and $r \in F(R)$. By "$R$ is $F$-pure," we mean that $R \subseteq F(R)$ is pure. Recall that
$
R^\infty := \varinjlim \bigl( R \xrightarrow{F} R \xrightarrow{F} R \xrightarrow{F} \cdots \bigr).
$
The following completes the proof of Theorem~\ref{18}.

\begin{proposition}\label{5.15}
	Suppose $R$ is a $d$-dimensional $F$-pure ring with $\operatorname{char} R = p > 0$ and
	$
	\operatorname{Ext}_R^i(R/\mathfrak{m}, R^\infty) = 0
	$
	for some $i > d$. Then $R$ is Gorenstein. In particular, $R$ is regular provided it is of isolated singularity.
\end{proposition}

\begin{proof}
	Since purity behaves well with respect to direct limits, and $R$ is $F$-pure, we see that $0 \to R \to R^\infty$ is pure. Now, the above fact implies that
$
	\operatorname{Ext}_R^i(R/\mathfrak{m}, R) \subseteq \operatorname{Ext}_R^i(R/\mathfrak{m}, R^\infty).
$
	Hence $\operatorname{Ext}_R^i(R/\mathfrak{m}, R) = 0$. Again, by \cite{Ro2}, $R$ is Gorenstein. In particular, it is Cohen--Macaulay. This, together with $\operatorname{Ext}_R^i(k, R^\infty) = 0$, enables us to use \cite[page 226]{sc3} and deduce that $R^\infty$ has finite injective dimension.
By the Gorenstein property, we see that $R^\infty$ has finite flat dimension. In particular, $\operatorname{Tor}_i^R(R/\mathfrak{m}, R^\infty) = 0$ for some $i$. By \cite[3.6]{ABEr}, $R$ is regular.
\end{proof}

\begin{remark}
For example
any  Stanley-Reisner ring
is F-pure, and any F-pure ring is reduced. In particular, if $R$ is 1-dimensional F-pure and local,
then $R$ is of isolated singularity.
\end{remark}
Let us remove the $F$-purity assumption from Proposition~\ref{5.15}.

\begin{proposition}\label{517}
	Suppose $R$ is a $d$-dimensional ring with isolated singularity, $\operatorname{char}(R) = p > 0$, and
	$
	\operatorname{Ext}_R^i(R/\mathfrak{m}, R^\infty) = 0
	$
	for some $i > d$. Then $R$ is regular.
\end{proposition}

\begin{proof}
	The result of \cite[1.2]{IYENGAR} applied to $\operatorname{Ext}_R^i(R/\mathfrak{m}, R^\infty) = 0$ yields that
	$
	\operatorname{Ext}_R^j(R/\mathfrak{m}, R^\infty) = 0
	$
	for all $j \ge i$. This enables us to use \cite[Theorem VI.9]{sc3} and deduce that $R^\infty$ has finite injective dimension. By \cite[Example 3.17(i)]{ryo}, the ring $R$ is regular, as claimed.
\end{proof}

Let $\mathfrak{p} \in \operatorname{Spec}(R)$. By $k(\mathfrak{p})$, we mean the residue field of $R_{\mathfrak{p}}$. Similar to {Proposition} \ref{517}, we can show the following.

\begin{corollary}\label{isom}
	Let $(R, \mathfrak{m})$ be a ring of prime characteristic. Suppose that for some $i > \dim(R)$ we have
	$
	\operatorname{Ext}_{R}^i(k(\mathfrak{p}), (R^\infty)_{\mathfrak{p}}) = 0
	$
	for all $\mathfrak{p}$ in the singular locus of $R$. Then $R$ is regular.
\end{corollary}

\begin{conjecture}
Suppose $R$ is $d$-dimensional local ring and $\operatorname{char} (R) = p > 0$ and $\operatorname{Ext}_R^i(R/\mathfrak{m}, R^\infty) = 0$ for some $i>d$. Then $R$ is regular.
\end{conjecture}

	\begin{acknowledgement}
We thank Shravan Patankar for his useful comments and valuable remarks.
	\end{acknowledgement}

\end{document}